\documentclass[11pt,reqno]{amsart}

\usepackage{amsmath, amsthm, amsopn, amssymb}
\usepackage{enumerate}
\usepackage{hyperref}

\usepackage{tikz}
\usetikzlibrary{calc}

\newtheorem{thm}{Theorem}[section]
\newtheorem{claim}[thm]{Claim}
\newtheorem{conj}[thm]{Conjecture}
\newtheorem{cor}[thm]{Corollary}
\newtheorem{lemma}[thm]{Lemma}
\newtheorem{obs}[thm]{Observation}
\newtheorem{prob}[thm]{Problem}
\newtheorem{quest}[thm]{Question}

\theoremstyle{definition}
\newtheorem{example}[thm]{Example}

\newtheorem{case}{Case}
\newtheorem{subcase}{Case}

\renewcommand{\le}{\leqslant}
\renewcommand{\ge}{\geqslant}

\newcommand{\FF}{\mathcal{F}}

\newcommand{\les}{\le_\sigma}
\newcommand{\Ps}{P_\sigma}

\newcommand{\height}{h}
\newcommand{\width}{w}

\title{On the number of monotone sequences}

\date{\today}

\author{Wojciech Samotij}
\address{School of Mathematical Sciences, Tel Aviv University, Tel Aviv 69978, Israel; and Trinity College, Cambridge CB2 1TQ, UK}
\email{samotij@post.tau.ac.il}

\author{Benny Sudakov}
\address{Department of Mathematics, ETH, 8092 Zurich, Switzerland}
\email{benjamin.sudakov@math.ethz.ch}

\thanks{Research supported in part by: (WS) Trinity College JRF and 
Institute for Mathematical Research (FIM), ETH Z\"urich; (BS) SNSF grant 200021-149111 and by a USA-Israel BSF grant.}

\begin{document}

\maketitle

\begin{abstract}
One of the most classical results in Ramsey theory is the theorem of Erd{\H{o}}s and Szekeres from 1935, which says that every sequence of more than $k^2$ numbers contains a monotone subsequence of length $k+1$. We address the following natural question motivated by this result: Given integers $k$ and $n$ with $n \ge k^2+1$, how many monotone subsequences of length $k+1$ must every sequence of $n$ numbers contain? We answer this question precisely for all sufficiently large $k$ and $n \le k^2 + c 
k^{3/2} / \log k$, where $c$ is some absolute positive constant.
\end{abstract}

\section{Introduction}

\label{sec:introduction}

A typical problem in extremal combinatorics has the following form: What is the largest size of a structure which does not contain any forbidden configurations? Once this extremal value is known, it is very natural to ask how many forbidden configurations one is guaranteed to find in every structure of a certain size which is larger than the extremal value. There are many results of this kind. Most notably, there is a very large body of work on the problem of determining the smallest number of $k$-vertex cliques in a graph with $n$ vertices and $m$ edges, attributed to Erd{\H{o}}s and Rademacher; see~\cite{Er62, Er69, ErSi83, LoSi83, Ni11, Ra08, Re12}. In extremal set theory there, is an extension of the celebrated Sperner's theorem, where one asks for the minimum number of chains in a family of subsets of $\{1, \ldots, n\}$ with more than $\binom{n}{\lfloor n/2 \rfloor}$ members; see~\cite{DaGaSu13, DoGrKaSe13, ErKl74, Kl68}. Another example is a recent work in \cite{DaGaSu14}, motivated by the classical theorem of Erd\H{o}s, Ko, and Rado. It studies how many disjoint pairs must appear in a $k$-uniform set system of certain size.

One can ask analogous questions in Ramsey theory. Once we know the maximum size of a~structure which does not contain some unavoidable pattern, we may ask how many such patterns are bound to appear in every structure whose size exceeds this maximum. This direction of research has also been explored in the past. For example, a well-known problem of Erd{\H{o}}s is to determine the minimum number of monochromatic $k$-vertex cliques in a $2$-coloring of the edges of $K_n$; 
see, e.g.,~\cite{Co12, FrRo93, Th89}. This may be viewed as a natural extension of Ramsey's theorem.

In this paper, we consider a similar generalization of another classical result in Ramsey theory, the famous theorem of Erd{\H{o}}s and Szekeres~\cite{ErSz35}, which states that for every positive integer $k$, any sequence of more than $k^2$ numbers contains a monotone (that is, monotonically increasing or monotonically decreasing) subsequence of length $k+1$. To be more precise, we shall be interested in the following very natural problem.

\begin{prob}
  \label{prob:main}
  For every $k$ and $n$, determine the minimum number of monotone subsequences of length $k+1$ in a sequence of $n$ numbers.
\end{prob}

It is not clear when Problem~\ref{prob:main} was originally posed. It appears first in print in a paper of Myers~\cite{My02}, who attributes it to Albert, Atkinson, and Holton. It follows from the aforementioned theorem of Erd{\H{o}}s and Szeker{\'e}s that every sequence of $n$ numbers contains at least $n-k^2$ monotone subsequences of length $k+1$. When $n \le k^2+k$, this is easily seen to be sharp by considering a sequence built from $k$ increasing sequences of lengths $k$ or $k+1$ by concatenating them in decreasing order, such as the sequence $\tau_{k,n}$ defined below. Without loss of generality, we may restrict our attention to sequences that are permutations of the set $\{1, \ldots, n\}$, which we shall from now on abbreviate by~$[n]$.

Let us denote by $S_n$ the set of all permutations of $[n]$. Following~\cite{My02}, given a permutation $\sigma \in S_n$, we let $m_k(\sigma)$ denote the number of monotone subsequences of length $k+1$ in $\sigma$ and let
\[
m_k(n) = \min\{m_k(\sigma) \colon \sigma \in S_n \}.
\]
In order to give an upper bound on $m_k(n)$, consider the permutation $\tau_{k,n}$ described by
\[
\begin{array}{llll}
  \lfloor (k-1)n/k \rfloor + 1, &  \lfloor (k-1)n/k \rfloor + 2, & \ldots, & n \\
  \lfloor (k-2)n/k \rfloor + 1, &  \lfloor (k-2)n/k \rfloor + 2, & \ldots, & \lfloor (k-1)n/k \rfloor \\
  \vdots & \vdots & & \vdots \\
  1, & 2, & \ldots, & \lfloor n/k \rfloor.
\end{array}
\]
Let $r_{k,n}$ be the unique number $r \in \{0, \ldots, k-1\}$ satisfying $r \equiv n \pmod k$. Since $\tau_{k,n}$ contains no decreasing subsequences of length $k+1$, it is easy to see that
\begin{equation}
  \label{eq:mktau}
  m_k(\tau_{k,n}) = r_{k,n} \binom{\lceil n/k \rceil}{k+1} + (k-r_{k,n}) \binom{\lfloor n/k \rfloor}{k+1}.
\end{equation}

It seems quite natural to guess that $m_k(n) = m_k(\tau_{k,n})$ for all $k$ and $n$, that is, that $\tau_{k,n}$ contains the minimum number of monotone subsequences of length $k+1$ among all permutations of $[n]$. This was conjectured by Myers~\cite{My02}, who noticed that Goodman's formula~\cite{Go59}, indeed proves that $m_2(n) = m_2(\tau_{2,n})$ for all $n$ and yields a characterisation of all permutations achieving equality.

\begin{conj}[Myers~\cite{My02}]
  \label{conj:main}
  Let $n$ and $k$ be positive integers. In any permutation of $[n]$, there are at least $m_k(\tau_{k,n})$ monotone subsequences of length $k+1$.
\end{conj}
Very recently, Balogh et al.~\cite{BaHuLiPiUdVo} proved this conjecture for $k=3$ and sufficiently large $n$ and described all extremal permutations. Their proof uses computer assistance and is based on the framework of flag algebras.

\subsection{Our results}

In this paper, we provide first evidence supporting Conjecture~\ref{conj:main} for large $k$. Our main result is that the conjecture holds for all sufficiently large $k$, 
as long as $n$ is not much larger than $k^2$. This provides a good understanding of how the minimum number of monotone subsequences of length $k+1$ grows in a short interval above the extremal threshold. Our results are similar in spirit to the ones of~\cite{Er62, LoSi83}, which determine the minimum number of cliques in graphs whose number of edges is slightly supercritical (larger than the Tur{\'a}n number for the clique).

\begin{thm}
  \label{thm:main}
  There exist an integer $k_0$ and a positive real $c$ such that $m_k(n) = m_k(\tau_{k,n})$ for all $k$ and $n$ satisfying $k \ge k_0$ and $n \le k^2 + c k^{3/2}/\log k$. Moreover, if $n \neq k^2+k+1$ and $m_k(\sigma) = m_k(n)$ for some $\sigma \in S_n$, then $\sigma$ contains monotone subsequences of length $k+1$ of only one type (increasing or decreasing).
\end{thm}

\begin{figure}[h]
  \centering
  \begin{tabular}{ccccc}
    \begin{tikzpicture}
      \def\s{0.35}
      \def\sigma{
        {9, 10, 11, 12, 13, 5, 6, 7, 8, 1, 2, 3, 4}
      }
      \draw[step=\s,black,very thin] (0,0) grid (12*\s,12*\s);
      \foreach \x in {0, ..., 12}
      {
        \filldraw (\x*\s, \sigma[\x]*\s-\s) circle[radius=2pt];
      }
    \end{tikzpicture}

    &

    \quad

    &

    \begin{tikzpicture}
      \def\s{0.35}
      \def\sigma{
        {10, 6, 11, 12, 13, 2, 7, 8, 9, 1, 3, 4, 5}
      }
      \draw[step=\s,black,very thin] (0,0) grid (12*\s,12*\s);
      \foreach \x in {0, ..., 12}
      {
        \filldraw (\x*\s, \sigma[\x]*\s-\s) circle[radius=2pt];
      }
    \end{tikzpicture}

    &

    \quad

    &

    \begin{tikzpicture}
      \def\s{0.35}
      \def\sigma{
        {10, 6, 11, 12, 13, 3, 7, 8, 9, 1, 2, 4, 5}
      }
      \draw[step=\s,black,very thin] (0,0) grid (12*\s,12*\s);
      \foreach \x in {0, ..., 12}
      {
        \filldraw (\x*\s, \sigma[\x]*\s-\s) circle[radius=2pt];
      }
    \end{tikzpicture}
  \end{tabular}
  \caption{Permutations $\tau_{3,13}$, $\sigma_3^1$, and $\sigma_3^2$.}
  \label{fig:sigma-ik}
\end{figure}
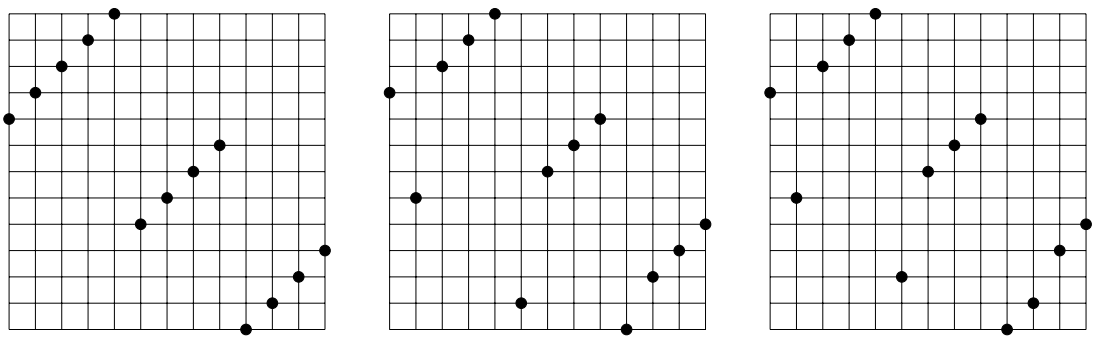

Somewhat surprisingly, if $n = k^2+k+1$, then there are $\sigma \in S_n$ with $m_k(\sigma) = m_k(n) = 2k+1$ which contain both increasing and decreasing subsequences of length $k+1$. Two such permutations are $\sigma_k^1$ and $\sigma_k^2$, where $\sigma_k^i$ is described by
\[
\begin{array}{llllll}
k^2+1, & k^2-k, & k^2+2, & k^2+3, & \ldots, & k^2+k+1, \\
& k^2-2k-1, & k^2-k+1, & k^2-k+2, & \ldots, & k^2,\\
& \vdots & \vdots & \vdots & & \vdots \\
& k+3, & 2k+5, & 2k+6, & \ldots, & 3k+4, \\
& 1+i, & k+4, & k+5, & \ldots, & 2k+3,\\
& 1, & 4-i, & 4, & \ldots, & k+2.
\end{array}
\]
One can check that $\sigma_k^i$ contains $2k+1-i$ increasing subsequences of length $k+1$ and $i$ decreasing subsequences of length $k+1$, see Figure~\ref{fig:sigma-ik}. However, no permutation $\sigma$ with $m_k(\sigma) = m_k(n) = 2k+1$ can have more monotone subsequences of length $k+1$ of the `odd' type than $\sigma_k^2$. It will follow from our proof of Theorem~\ref{thm:main} that for each extremal permutation, at least $2k-1$ out of its $2k+1$ monotone subsequences of length $k+1$ are of the same type. For details, we refer the reader to Theorem~\ref{thm:posets} and Example~\ref{example:extremal-posets}.

\subsection{Chains and antichains in posets}

Every permutation admits a natural representation as a poset (partially ordered set) in which its increasing and decreasing subsequences are mapped to chains and antichains, respectively. Indeed, given a permutation $\sigma$ of $[n]$, one may define a binary relation $\les$ on $[n]$ by letting $i \les j$ if and only if $i \le j$ and $\sigma(i) \le \sigma(j)$. It is not hard to see that $\Ps = ([n], \les)$ is a poset whose chains and antichains are in a one-to-one length-preserving\footnote{By `length' of a chain or an antichain we mean the number of its elements.} correspondence with increasing and decreasing subsequences in $\sigma$, respectively. Via this correspondence, one may easily deduce the theorem of Erd{\H{o}}s and Szekeres from the famous theorem of Dilworth~\cite{Di50}, which says that every poset containing no antichain with $k+1$ elements admits a partition into $k$ chains, or its much easier to prove dual version (due to Mirsky~\cite{Mi71}), which says that every poset containing no chain with $k+1$ elements can be partitioned into $k$ antichains, see Section~\ref{sec:decomposition}.

Let us call a set $A$ of elements of a poset \emph{homogenous} if $A$ is a chain or an antichain. A natural generalization of Problem~\ref{prob:main} to posets would be the following.

\begin{prob}
  \label{prob:posets}
  For every $k$ and $n$, determine the minimum number of homogenous $(k+1)$-element sets in a poset with $n$ elements.
\end{prob}

Given a poset $P$, we let $h_k(P)$ denote the number of homogenous sets of cardinality $k+1$ in $P$ and
\[
h_k(n) = \min\{h_k(P) \colon \text{$P$ is a poset with $n$ elements}\}.
\]
It follows from the above discussion that $h_k(n) \le m_k(n)$ and we think that it is natural to ask the following.

\begin{quest}
  Is it true that $h_k(n) = m_k(n)$ for all $n$ and $k$?
\end{quest}

Clearly, not every poset is isomorphic to $\Ps$ for some permutation $\sigma$. In fact, this is the case precisely for posets of \emph{order dimension} at most two, that is, posets that are the intersection of two linear orders. Nevertheless, more for the sake of convenience rather than generality, we shall present our arguments using the language of posets. We remark here that several times in the proof, we will use the fact that we can `flip' our poset $P$, exchanging the roles of chains and antichains. That is, we will assume that there is a dual poset $P^*$ defined on the same set as $P$ such that every pair of elements is comparable in either $P$ or $P^*$ but not both of them. This is possible only for posets of order dimension at most two, that is, ones that represent permutations. Indeed, for every permutation $\sigma \in S_n$, we have $P_\sigma^* = P_{\sigma^*}$, where $\sigma^*(i) = n+1-\sigma(i)$. The converse statement was proved by Dushnik and Miller~\cite{DuMi41}. Let us now rephrase Theorem~\ref{thm:main} in the language of posets.

\begin{thm}
  \label{thm:posets}
  There exist an integer $k_0$ and a positive real $c$ such that the following is true. Let $k$ and $n$ be integers satisfying $k \ge k_0$ and $n \le k^2 + c k^{3/2} / \log k$. If $P$ is an $n$-element poset of order dimension at most two, then
  \begin{equation}
    \label{eq:hkP-lower}
    h_k(P) \ge m_k(\tau_{k,n}).
  \end{equation}
  Moreover, if equality holds in~\eqref{eq:hkP-lower}, then $P$ can be decomposed into $k$ chains or $k$ antichains of length $\lfloor n/k \rfloor$ or $\lceil n/k \rceil$ each, unless $n = k^2 + k + 1$ and $P$ (or $P^*$) is one of the posets described in Example~\ref{example:extremal-posets} below.
\end{thm}

\begin{example}
  \label{example:extremal-posets}
  Suppose that $n = k^2+k+1$ and observe that $m_k(\tau_{k,n}) = 2k+1$. We describe two families of $n$-element posets with exactly $2k+1$ homogenous $(k+1)$-sets that contain both chains and antichains with $k+1$ elements. Each of these posets has precisely $k+1$ minimal elements; denote the set of these minimal elements by $A_1$. Moreover, $P \setminus A_1$ can be decomposed into $k$ chains as well as into $k$ antichains. In particular, each chain and each antichain in every such decomposition has precisely $k$ elements. Furthermore, $P \setminus A_1$ contains only $k$ chains of length $k$. Let $A_2$ be the set of minimal elements of $P \setminus A_1$ and note that $|A_2| = k$. The comparability graph of the subposet of $P$ induced by $A_1 \cup A_2$ is either (i) a path with $2k+1$ vertices or (ii) the disjoint union of a path with $2k-1$ vertices and an edge. Moreover, if (ii) holds, then one of the elements of $A_1$ belonging to the path of length $2k-1$ is smaller than the second smallest element of the unique $k$-element chain in $P \setminus A_1$ whose smallest element is the unique element of $A_2$ that does not belong to the path. One can check that if (i) holds, then $P$ has precisely $2k$ chains and one antichain with $k+1$ elements and that if (ii) holds, then $P$ has precisely $2k-1$ chains and $2$ antichains with $k+1$ elements. Finally, in both cases, there exist posets of order dimension at most two fitting the description. Two examples of such posets are $P_{\sigma_k^1}$ and $P_{\sigma_k^2}$, where $\sigma_k^1$ and $\sigma_k^2$ are the permutations defined below Theorem~\ref{thm:main}; see Figure~\ref{fig:posets-sigma-ik}. 
\end{example}

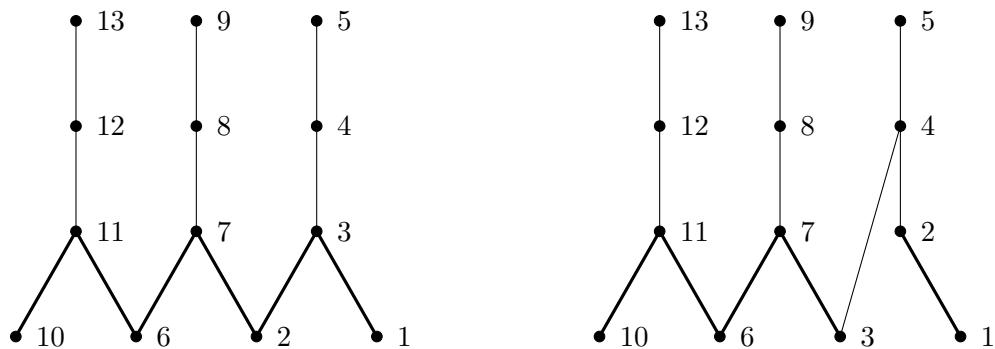
\begin{figure}[h]
  \centering
  
  \begin{tabular}{ccc}
    \begin{tikzpicture}
      \def\s{0.8}
      \def\t{1.4}
      \def\p{
        {10,6,2,1,11,7,3,12,8,4,13,9,5}
      }

      \foreach \x [evaluate=\x as \lab using ({\p[\x]})] in {0, ..., 3}
      {
        \node [label=right:$\lab$] at (2*\x*\s, 0) {};
        \filldraw (2*\x*\s, 0) circle[radius=2pt];
      }
      
      \foreach \x in {0, ..., 2}
      {
        \draw [very thick] (2*\x*\s, 0) -- (2*\x*\s + \s, \t) -- (2*\x*\s + 2*\s, 0);
        \draw (2*\x*\s + \s, \t) -- (2*\x*\s + \s, 3*\t);
      }
      
      \foreach \y in {1, ..., 3}
      {
        \foreach \x [evaluate=\x as \lab using ({\p[3*\y+\x+1]})] in {0, ..., 2}
        {
          \node [label=right:$\lab$] at (2*\x*\s+\s, \y*\t) {};
          \filldraw (2*\x*\s+\s, \y*\t) circle[radius=2pt];
        }
      }
    \end{tikzpicture}

    &

    \qquad\qquad\qquad

    &

    \begin{tikzpicture}
      \def\s{0.8}
      \def\t{1.4}
      \def\p{
        {10,6,3,1,11,7,2,12,8,4,13,9,5}
      }

      \foreach \x [evaluate=\x as \lab using ({\p[\x]})] in {0, ..., 3}
      {
        \node [label=right:$\lab$] at (2*\x*\s, 0) {};
        \filldraw (2*\x*\s, 0) circle[radius=2pt];
      }
      
      \foreach \x in {0, ..., 1}
      {
        \draw [very thick] (2*\x*\s, 0) -- (2*\x*\s + \s, \t) -- (2*\x*\s + 2*\s, 0);
      }
      
      \draw (4*\s, 0) -- (5*\s, 2*\t);
      \draw [very thick] (5*\s, \t) -- (6*\s, 0);
      
      \foreach \x in {0, ..., 2}
      {
        \draw (2*\x*\s + \s, \t) -- (2*\x*\s + \s, 3*\t);
      }
      
      \foreach \y in {1, ..., 3}
      {
        \foreach \x [evaluate=\x as \lab using ({\p[3*\y+\x+1]})] in {0, ..., 2}
        {
          \node [label=right:$\lab$] at (2*\x*\s+\s, \y*\t) {};
          \filldraw (2*\x*\s+\s, \y*\t) circle[radius=2pt];
        }
      }
    \end{tikzpicture}
  \end{tabular}

  \caption{Hasse diagrams of posets $P_{\sigma_3^1}$ and $P_{\sigma_3^2}$. The edges of the comparability graph induced by $A_1 \cup A_2$ are thickened.}
  \label{fig:posets-sigma-ik}
\end{figure}

\subsection{Outline of the paper}
The remainder of the paper is organized as follows. In Section~\ref{sec:decomposition}, we describe a canonical decomposition of an arbitrary poset into antichains and introduce several pieces of notation used in the proof of Theorem~\ref{thm:posets} and in Section~\ref{sec:auxiliary-lemmas}, we collect several auxiliary lemmas. In Section~\ref{sec:outline-proof}, we present a brief outline of the proof of Theorem~\ref{thm:posets}. Section~\ref{sec:large-surplus} is devoted to the proof of one of our main lemmas, which provides a lower bound on the number of homogenous sets in posets with large `surplus' (this notion will be defined in Section~\ref{sec:outline-proof}). Finally, Section~\ref{sec:proof} contains the proof of Theorem~\ref{thm:posets}. We close the paper with several concluding remarks.

\section{Decomposition into antichains}

\label{sec:decomposition}

In our arguments, we shall rely on the following canonical decomposition of an arbitrary poset into antichains, cf.\ Mirsky's theorem~\cite{Mi71}. Fix an arbitrary poset $P$. Recall that the height and the width of $P$, which we shall denote by $\height(P)$ and $\width(P)$, are the cardinalities of the largest chain and the largest antichain in $P$, respectively. For each positive integer $i$, let
\[
A_i = \{x \in P \colon \text{the longest chain $L$ with $\max L = x$ has $i$ elements}\}.
\]
In other words, $A_1$ is the set of minimal elements of $P$ and for every $i \ge 1$, $A_{i+1}$ is the set of minimal elements in $P \setminus (A_1 \cup \ldots \cup A_i)$. For each $i$, the set $A_i$ is an antichain and thus $|A_i| \le \width(P)$, and
\[
P = \bigcup_{i=1}^{\height(P)} A_i.
\]

For each $i$ with $1 \le i < \height(P)$, let $G_i$ be the bipartite graph on the vertex set $A_i \cup A_{i+1}$ whose edges are all pairs $xy$ with $x \in A_i$ and $y \in A_{i+1}$ such that $x \le y$. In other words, $G_i$ is the Hasse diagram of the subposet of $P$ induced by $A_i \cup A_{i+1}$. Observe that each vertex in $A_{i+1}$ has at least one $G_i$-neighbor in $A_i$ (as otherwise it would belong to $A_1 \cup \ldots \cup A_i$). On the other hand, it is possible that some vertices in $A_i$ have degree zero in $G_i$ (as clearly not all elements of $P$ have to belong to some chain of maximum length).

Let $h = \height(P)$. For every $i \in [h]$ and $x \in A_i$, we let $u_i(x)$ be the number of chains $L \subseteq P$ of length $h - i + 1$ with $\min L = x$. Observe that $u_h(x) = 1$ for every $x \in A_h$ and that for all $i \in [h-1]$ and $x \in A_i$,
\[
u_i(x) = |\{(x_{i+1}, \ldots, x_h) \in A_{i+1} \times \ldots \times A_h \colon x \le x_{i+1} \le \ldots \le x_h\}|.
\]
Upon making this definition, one easily verifies that the number of chains of length $h$ in $P$ is $\sum_{x \in A_1} u_1(x)$ and that for each $i \in [h-1]$ and $x \in A_i$,
\begin{equation}
  \label{eq:ui-uip}
  u_i(x) = \sum_{xy \in G_i} u_{i+1}(y).
\end{equation}
Since no element $x \in A_i$ with $u_i(x) = 0$ will be contributing anything to the total count of chains of length $h$, we shall often be focusing our attention on the set $A_i' \subseteq A_i$ defined by
\[
A_i' = \{x \in A_i \colon u_i(x) \ge 1\}.
\]
Let us note here for future reference that~\eqref{eq:ui-uip} implies that there are no edges of $G_i$ between $A_{i+1}'$ and $A_i \setminus A_i'$. As we shall often estimate the sum of $u_i(x)$ over all $x \in A_i$, let us abbreviate it by $\Sigma_i$. That is, for each $i \in [h]$, let
\[
\Sigma_i = \sum_{x \in A_i} u_i(x).
\]
Since each $y \in A_{i+1}$ has at least one $G_i$-neighbor in $A_i$, it follows from~\eqref{eq:ui-uip} that
\begin{equation}
  \label{eq:sum-ui-monotone}
  \Sigma_{i} \ge \Sigma_{i+1}.
\end{equation}
Clearly, \eqref{eq:sum-ui-monotone} holds with equality if and only if each $y \in A_{i+1}'$ has exactly one $G_i$-neighbor in $A_i$. This naturally leads to the final definition of this section. Namely, for each $i \in [h-1]$, we let
\[
B_{i+1} = \{y \in A_{i+1}' \colon \deg_{G_i}(y) = 1\}.
\]

\section{Auxiliary lemmas}

\label{sec:auxiliary-lemmas}

In this section, we collect a few auxiliary lemmas that will be repeatedly used in the proof of Theorem~\ref{thm:posets}. Our first lemma is a straightforward corollary of the Kruskal--Katona theorem~\cite{Ka68, Kr63}.

\begin{lemma}
  \label{lemma:KK}
  Suppose that $a \ge b > 0$, let $\FF$ be an arbitrary family of $a$-element sets, and define
  \[
  \partial_b \FF = \{B \colon \text{$|B| = b$ and $B \subseteq A$ for some $A \in \FF$}\}.
  \]
  Then
  \[
  |\partial_b \FF| \ge \min \{ |\FF|/2, 2^b \}.
  \]
\end{lemma}
\begin{proof}
  If $\FF = \emptyset$, $a=b$, or $a \ge 2b$, then the assertion of the lemma is trivial as
  \begin{equation}
    \label{eq:a-choose-b}
    \binom{2b}{b} \ge 2^b.
  \end{equation}
  Otherwise, let $m$ be the smallest integer such that $\binom{m}{a} > |\FF|$. By the Kruskal--Katona theorem~\cite{Ka68,Kr63}, $|\partial_b \FF| \ge \binom{m-1}{b}$. If $m-1 \ge 2b$, then the conclusion follows from~\eqref{eq:a-choose-b}. Otherwise, since $b < a \le m \le 2b$, then $\binom{m}{a} \le \binom{m}{b+1}$ and consequently,
  \[
  |\partial_b \FF| \ge \binom{m-1}{b} > \binom{m-1}{b} \frac{|\FF|}{\binom{m}{a}} \ge \frac{\binom{m-1}{b}}{\binom{m}{b+1}} |\FF| = \frac{b+1}{m} |\FF| \ge \frac{|\FF|}{2}.\qedhere
  \]
\end{proof}

Our second lemma will be essential in proving a lower bound on the number of chains of length $k+1$ in a poset of height larger than $k+1$ in terms of the number of chains of maximum length. The lemma is somewhat abstract, but it is immediately followed by a much more concrete corollary.

\begin{lemma}
  \label{lemma:signatures}
  Suppose that $M$ is a positive integer, $X$ and $Y$ are arbitrary sets, and $f_1, \ldots, f_M \colon X \to Y$ are pairwise different functions. There exist sets $X_1, \ldots, X_M \subseteq X$ with $|X_i| \le \log_2 M$ for all $i \in [M]$ such that
  \begin{equation}
    \label{eq:fifj}
    f_i|_{X_i \cup X_j} \neq f_j|_{X_i \cup X_j} \quad \text{for all $i \neq j$}.
  \end{equation}
\end{lemma}
\begin{proof}
  We prove the lemma by induction on $M$. The statement is trivial if $M=1$ (one takes $I_1 = \emptyset$), so we may assume that $M \ge 2$. Since $f_1, \ldots, f_M$ are pairwise different, there is an $x \in X$ such that not all $f_i$ take the same value at $x$. For each $y \in Y$, let
  \[
  I(x, y) = \{i \in [M] \colon f_i(x) = y\}
  \]
 and let $y \in Y$ be a value that maximizes $|I(x, y)|$. Note that $|I(x, y)| < M$ by our choice of $x$ and that $|I(x, z)| \le M/2$ for each $z \in Y \setminus \{y\}$. We apply the inductive assumption separately to $\{f_i \colon i \in I(x, z)\}$ for each $z \in Y$ with $I(x,z) \neq \emptyset$ to obtain sets $X_1', \ldots, X_M' \subseteq X$ such that $|X_i'| < \log_2M$ for every~$i$, $|X_i'| \le \log_2 (M/2)$ for every $i \notin I(x,y)$, and~\eqref{eq:fifj} holds for each pair $\{i,j\}$ which is fully contained in one of the sets $I(x,z)$. It is straightforward to check that the sets $X_1, \ldots, X_M$ defined by
 \[
 X_i =
 \begin{cases}
   X_i' & \text{if $i \in I(x,y)$,} \\
   X_i' \cup \{x\} & \text{if $i \notin I(x, y)$,}
 \end{cases}
 \]
 satisfy the assertion of the lemma.
\end{proof}

\begin{cor}
  \label{cor:signatures}
  Let $k$, $\ell$, and $M$ be positive integers, let $P$ be a poset of height $k+\ell$, and suppose that $m := \log_2 M + 1 \le k/4$.
  \begin{enumerate}[(i)]
  \item
    \label{item:signatures-global}
    If $P$ contains at least $M$ chains of length $k+\ell$, then it contains at least
    \[
    \exp \left( - \frac{2(\ell-1)m}{k} \right) \cdot M \binom{k+\ell}{k+1}
    \]
    chains of length $k+1$.
  \item
    \label{item:signatures-local}
    Given any $y \in P$, (\ref{item:signatures-global}) still holds if we replace `chains' with `chains containing $y$'.
  \end{enumerate}
\end{cor}
\begin{proof}
  We prove (\ref{item:signatures-global}) and (\ref{item:signatures-local}) simultaneously. Suppose that $L_1, \ldots, L_M$ are pairwise distinct chains of length $k+\ell$. For (\ref{item:signatures-local}), assume moreover that each $L_i$ contains $y$. Viewing each chain $L_i$ as a function from $[k+\ell]$ to $P$, we invoke Lemma~\ref{lemma:signatures} to obtain sets $X_1, \ldots, X_M \subseteq [k+\ell]$ such that $|X_i| \le \log_2M$ for each $i$ and $L_i(X_i \cup X_j) \neq L_j(X_i \cup X_j)$ whenever $i \neq j$. For (\ref{item:signatures-local}), add to each $X_i$ the unique index $x$ such that $L_i(x) = y$. By the definition of $m$, we have that $|X_i| \le m$ for each $i$.

  Let $N$ denote the number of chains of length $k+1$ that are obtained by fixing an $i \in [M]$ and an arbitrary $(k+1)$-set $C_i \subseteq [k+\ell]$ such that $X_i \subseteq C_i$ and considering the set $L_i(C_i)$. Note crucially that for any $i \neq j$ and any choice of $C_i$ and $C_j$ as above, the sets $L_i(C_i)$ and $L_j(C_j)$ are \emph{different} chains (containing $y$). Indeed, since $L_i$ and $L_j$ are chains of maximum length, then for every $z \in L_i \cap L_j$, there is a unique $x \in [k+\ell]$ such that $L_i(x) = z = L_j(x)$. Hence,
  \[
  \begin{split}
    N & = \sum_{i=1}^M \binom{|L_i| - |X_i|}{k+1 - |X_i|} = \sum_{i=1}^M \binom{k+\ell - |X_i|}{\ell - 1} \ge M \binom{k+\ell-m}{\ell-1} \\
    & \ge \left(1 - \frac{1}{k-m}\right)^{(\ell-1) m} \cdot M\binom{k+\ell}{\ell-1} \ge \exp\left(-\frac{2(\ell - 1)m}{k}\right) \cdot M \binom{k+\ell}{k+1}.
  \end{split}
  \]
  Above, we used the fact that $\binom{a}{b} \ge \left(1 - \frac{1}{a-b}\right)^b \binom{a+1}{b}$ and that $1-x \ge e^{-3x/2}$ if $x \le 1/3$.
\end{proof}

We close this section with a simple lower bound on the number of connected sets in trees. We shall use this bound in the analysis of one of the almost extremal cases in the proof of Theorem~\ref{thm:posets}.

\begin{lemma}
  \label{lemma:connected-sets}
  If $1 \le c \le t$, then every tree with $t$ vertices contains at least $t-c+1$ connected subsets with $c$ elements.
\end{lemma}
\begin{proof}
  We prove the statement by induction on $t-c$. It is certainly true if $t = c$. Assume that $t \ge c+1$ and let $T$ be a tree with $t$ vertices. Let $v$ be an arbitrary leaf of $T$ and set $T' = T - v$. By the inductive assumption, $T'$ contains at least $t-c$ connected subsets with $c$ elements. On the other hand, it is easy to check that $v$ is contained in at least one connected subset of $T$ of any given size between $1$ and $t$.
\end{proof}

\section{Outline of the proof}

\label{sec:outline-proof}

Roughly speaking, our proof of Theorem~\ref{thm:posets} is a combination of a stability-type argument and an induction on $n$. More precisely, given an $n$-element poset $P$, we either find an element $x \in P$ which belongs to at least $m_k(\tau_{k,n}) - m_k(\tau_{k,n-1})$ homogenous $(k+1)$-sets, in which case we may simply appeal to the inductive assumption on $P \setminus \{x\}$, or we show more `directly' that $P$ contains at least $m_k(\tau_{k,n})$ homogenous $(k+1)$-sets. Some extra work is needed to deduce the claimed structural description of $P$ when $h_k(P) = m_k(\tau_{k,n})$. At all times, we rely heavily on the assumption that the order dimension of $P$ is at most two and hence $P$, as well as each of its induced subposets, has a dual poset $P^*$. This assumption allows us to focus on counting chains, since we may always replace $P$ with $P^*$, exchanging the roles of chains and antichains. We shall tacitly assume that $\height(P) \ge \width(P)$, as $\height(P^*) = \width(P)$, and that $\height(P) < \lceil n/k \rceil$, as otherwise each element of a longest chain in $P$ belongs to at least $m_k(\tau_{k,n}) - m_k(\tau_{k,n-1})$ chains of length $k+1$.

We first show that if $P$ is `far' from being a union of $k$ chains (or $k$ antichains), then the number of homogenous $(k+1)$-sets is much greater than $m_k(\tau_{k,n})$. To this end, we define a simple parameter termed surplus which measures the distance between a poset and a union of $k$ chains. Let $P$ be an arbitrary poset of height $h$ and let $k$ be an integer. The \emph{$k$-surplus} of $P$, denoted by $s_k(P)$, is defined by $s_k(P) = n - hk$. Observe that
\begin{equation}
  \label{eq:surplus}
  s_k(P) = \sum_{i = 1}^h (|A_i| - k),
\end{equation}
where $(A_i)_{i=1}^h$ is the canonical decomposition of $P$ into antichains. In Section~\ref{sec:large-surplus}, we show that if $s_k(P) = \Omega(k)$, then $h_k(P) = 2^{\Omega(\sqrt{k})} \gg m_k(\tau_{k,n})$. On the other hand, $s_k(P) = o(k)$ together with $\height(P) < \lceil n/k \rceil$ imply that $h(P) = \lceil n/k \rceil - 1$.

In the remainder of the proof, we prove a sequence of lower bounds on $\Sigma_1$, the number of chains of maximum length. The proof of each of these bounds relies on the analysis of the graphs $G_i$ for various indices $i$ such that $A_i \cup A_{i+1}$ contains an antichain with $k+1$ elements. Roughly speaking, we show that for each such $i$, either $\Sigma_i - \Sigma_{i+1}$ is large or $A_i \cup A_{i+1}$ contains many $(k+1)$-element antichains. Each of these bounds gives some sufficient conditions on $P$ which imply that $h_k(P) > m_k(\tau_{k,n})$; here, we use Corollary~\ref{cor:signatures} to translate a lower bound on the number of chains of length $h(P)$ into a lower bound on the number of chains of length $k+1$. If $P$ does not satisfy any of these conditions, then the canonical decomposition of $P$ into antichains becomes greatly restricted. In particular, there are very few indices $i$ with $|A_i| > k$ and $|A_i'| \approx |A_i|$ for all $i$. Finally, some careful case analysis, which involves counting both chains and antichains with $k+1$ elements, shows that $h_k(P) \ge m_k(\tau_{k,n})$ and this inequality is strict unless $n = k^2+k+1$ and $P$ (or $P^*$) is one of the posets described in Example~\ref{example:extremal-posets}.

\section{Posets with large surplus}

\label{sec:large-surplus}

In this section, we prove one of our key lemmas. It says that posets with large $k$-surplus and no `bottlenecks' (small sets whose deletion reduces the height) contain many chains of maximum length or many $(k+1)$-element antichains.

\begin{lemma}
  \label{lemma:large-surplus}
  Let $d$, $k$, and $s$ be integers satisfying $1 \le d \le k$ and suppose that $P$ is a poset such that $s_k(P) \ge s$ and deletion of no $s/2$ elements reduces the height of $P$. Then $P$ contains either at least $2^d$ antichains with $k+1$ elements or at least $2^{\lfloor s/(2d) \rfloor}$ chains of length $\height(P)$.
\end{lemma}
\begin{proof}
  We fix $d$ and $k$ with $1 \le d \le k$ and prove the statement by induction on $s$. If $0 \le s < 2d$, then the assertion of the lemma holds vacuously. Suppose now that $s \ge 2d$ and that $P$ satisfies the assumptions of the lemma. Let $(A_j)_{j=1}^{h(P)}$ be the canonical decomposition of $P$ into antichains and let $i$ be the smallest index such that $|A_i| > k$; such $i$ exists since $s_k(P) > 0$, see~\eqref{eq:surplus}. Since $A_i$ is an antichain, we may assume that $|A_i| \le k+d$ since otherwise the number $N$ of $(k+1)$-element antichains in $A_i$ alone satisfies
  \[
  N \ge \binom{k+d+1}{k+1} = \binom{k+d+1}{d} = \prod_{j=1}^d \frac{k+j+1}{j} \ge 2^d,
  \]
  where the last inequality follows from our assumption that $d \le k$. Recall the definition of $G_i$ and let $B_{i+1}$ be the set of all $y \in A_{i+1}$ with at most (exactly) one $G_i$-neighbor in $A_i$. For every $Y \subseteq B_{i+1}$, the set $(A_i \setminus N_{G_i}(Y)) \cup Y$ is an antichain with at least $|A_i|$ elements and therefore we may further assume that $|B_{i+1}| < d$ as otherwise $A_i \cup B_{i+1}$ contains at least $2^d$ antichains of size $k+1$. To see this, for every $Y \subseteq B_{i+1}$ with $|Y| \le d$, consider an arbitrary $(k+1)$-element set $L$ with $Y \subseteq L \subseteq (A_i \setminus N_{G_i}(Y)) \cup Y)$. By~\eqref{eq:ui-uip} and~\eqref{eq:sum-ui-monotone},
  \begin{equation}
    \label{eq:chains}
    \sum_{x \in A_1} u_1(x) \ge \sum_{x \in A_i} u_i(x) = \sum_{y \in A_{i+1}} \deg_{G_i}(y) u_{i+1}(y) \ge 2 \sum_{y \in A_{i+1} \setminus B_{i+1}} u_{i+1}(y).
  \end{equation}
  
  Let $h = \height(P)$ and observe that the sum in the right-hand side of~\eqref{eq:chains} is precisely the number of chains of length $h-i$ in the poset $P'$ obtained from $P$ by deleting $A_1 \cup \ldots \cup A_i \cup B_{i+1}$. In order to estimate this sum, we shall apply the inductive assumption to $P'$. First, note that $\height(P') \le h-i$, as $A_{i+1} \setminus B_{i+1}, A_{i+1}, \ldots, A_h$ partition $P'$ into $h-i$ antichains. Moreover, if $\height(P' \setminus X) < h-i$ for some $X \subseteq P'$, then $\height(P \setminus (B_{i+1} \cup X)) < h$ as every chain in $P$ contains at most one element from each of $A_1, \ldots, A_i$. Therefore, not only $\height(P') = h-i$, as $|B_{i+1}| < d \leq s/2$, but also the deletion of no $s/2-d$ elements reduces the height of $P'$. Furthermore,
  \[
  \begin{split}
    s_k(P') & = |P'| - k(h-i) = |P| - \sum_{j=1}^i (|A_j| - k) - |B_{i+1}| - hk \\
    & \ge s_k(P) - (|A_i| - k) - |B_{i+1}| > s-2d.
  \end{split}
  \]
  The first inequality above follows from the minimality of $i$ and the second inequality from the assumptions that $|A_i| \le k+d$ and $|B_{i+1}| < d$. Hence, $P'$ satisfies the assumptions of the lemma with $s$ replaced by $s-2d$. This means that either $P'$ contains at least $2^d$ antichains of size $k+1$ or
  \[
  \sum_{y \in A_{i+1} \setminus B_{i+1}} u_{i+1}(y) \ge 2^{\lfloor s/(2d) \rfloor - 1}.
  \]
  By~\eqref{eq:chains}, this completes the proof of the lemma.
\end{proof}

\begin{cor}
  \label{cor:large-surplus}
  Let $k$ and $t$ be integers satisfying $0 < t \le k/2$ and suppose that $P$ is a poset of order dimension at most two such that $\height(P) \ge \width(P)$ and $s_k(P) \ge 3t$. Then $P$ contains at least $2^{\sqrt{t}-1}$ homogenous $(k+1)$-sets.
\end{cor}
\begin{proof}
  We first `prune' $P$ by repeatedly performing the following two-step procedure:
  \begin{enumerate}[(1)]
  \item
    \label{item:prune-1}
    If $P$ contains a set $S$ of at most $t$ elements whose deletion reduces the height of $P$, then remove the smallest such $S$ from $P$.
  \item
    \label{item:prune-2}
    If $\height(P) < \width(P)$, then replace $P$ with $P^*$, exchanging the roles of chains and antichains.
  \end{enumerate}
  Let us list several properties of the `pruning' procedure. First, performing (\ref{item:prune-1}) decreases the height of $P$ by exactly one at the cost of deleting at most $t$ elements. Thus, each time (\ref{item:prune-1}) is executed, the $k$-surplus of $P$ increases by at least $k-t$. Second, since in the beginning and after (\ref{item:prune-2}) is performed, $\height(P) \ge \width(P)$, step (\ref{item:prune-2}) can be executed only in conjunction with (\ref{item:prune-1}). Third, each time (\ref{item:prune-2}) is performed, it increases the height of $P$ by exactly one, thus reducing the $k$-surplus of $P$ by $k$. Moreover, this cannot happen in two consecutive rounds, since immediately after (\ref{item:prune-2}) is triggered, $\height(P) > \width(P)$.

  Therefore, letting $P'$ denote the final outcome of the `pruning' procedure and $r$ the number of rounds, we have (recall that $k \ge 2t$)
  \[
  \begin{split}
    s_k(P') & \ge s_k(P) + r(k-t) - \lceil r/2 \rceil k \ge 3t + \lfloor r / 2 \rfloor k - rt \\
    & \ge 3t + (r-1)t - rt = 2t > 0.
  \end{split}
  \]
  In particular,
  \[
  \height(P') \ge \width(P') \ge k + \frac{s_k(P')}{\height(P')} > k.
  \]
  Since $P'$ clearly does not contain a set of $t$ elements whose deletion reduces the height of $P'$, Lemma~\ref{lemma:large-surplus} with $d = \lfloor \sqrt{t} \rfloor$ and $s = 2t$ implies that $P'$ contains either at least $2^{\sqrt{t}-1}$ antichains of size $k+1$ or at least $2^{\sqrt{t}}$ chains of length $\height(P')$ and consequently, by Lemma~\ref{lemma:KK}, also at least $2^{\sqrt{t}-1}$ chains of length $k+1$. Finally, since $P$ contains either $P'$ or $(P')^*$, every homogenous set in $P'$ is also homogenous in $P$.
\end{proof}

\section{Proof of Theorem~\ref{thm:posets}}

\label{sec:proof}

Let $k$ be a sufficiently large integer. We prove the theorem by induction on $n$. The assertion is trivial when $n \le k^2$ as then $m_k(\tau_{k,n}) = 0$ and every poset with no chain of length $k+1$ can be covered by $k$ antichains, see Section~\ref{sec:decomposition}. Therefore, suppose that $k^2 < n \le k^2 c k^{3/2}/\log_2 k$, where $c = 1/300$, and let $P$ be an arbitrary $n$-element poset of order dimension at most two. Without loss of generality, we may assume that $\height(P) \ge \width(P)$ as otherwise we may replace $P$ by $P^*$, exchanging the roles of chains and antichains. Let $\ell$ and $q$ be the unique nonnegative integers satisfying $0 < q \le k$ and
\begin{equation}
  \label{eq:n-q-k-ell}
  n = q \cdot (k+\ell+1) + (k-q) \cdot (k+\ell).  
\end{equation}
In other words, we let $\ell = \lceil n / k \rceil - k - 1$ and $q = n - k(k+\ell)$. Our upper bound on $n$ implies that $\ell \le c\sqrt{k}/\log_2 k$, which we note here for future reference.

Observe that $m_k(\tau_{k,n})$ is precisely the number of chains of length $k+1$ in the poset which is the disjoint union of $k$ pairwise incomparable chains: $q$ chains of length $k+\ell+1$ and $k-q$ chains of length $k + \ell$, cf.~\eqref{eq:mktau}. In particular,
\begin{equation}
  \label{eq:mkn-mknm}
  m_k(\tau_{k,n}) - m_k(\tau_{k,n-1}) = \binom{k+\ell}{k}.
\end{equation}
With the view of~\eqref{eq:mkn-mknm}, we may and shall assume that $P$ contains no element $x$ that belongs to more than $\binom{k+\ell}{k}$ homogenous $(k+1)$-sets. Indeed, otherwise we could apply the inductive assumption to the poset $P \setminus \{x\}$ and conclude that $h_k(P) > m_k(\tau_{k,n})$.

\subsection{Posets of height larger than $k+\ell$}

Our inductive assumption allows us to easily deal with the case $\height(P) \ge k + \ell + 1$. Indeed, every element of each longest chain in $P$ lies in at least $\binom{\height(P)-1}{k}$ chains of length $k+1$ and hence by~\eqref{eq:mkn-mknm}, for any such element $x$,
\begin{equation}
  \label{eq:induction}
  h_k(P) \ge h_k(P \setminus \{x\}) + \binom{\height(P)-1}{k} \ge m_k(\tau_{k,n-1}) + \binom{k+\ell}{k} = m_k(\tau_{k,n}).
\end{equation}
Since we have promised to characterize all posets satisfying $h_k(P) = m_k(\tau_{k,n})$, we still need to analyze the case when all inequalities in~\eqref{eq:induction} are actually equalities. This means, in particular, that $\height(P) = k + \ell + 1$ and that no longest chain intersects a $(k+1)$-element antichain. We claim that $P$ may be partitioned into $k$ chains.

Let $x$ be the top element of some longest chain $L$ in~$P$. By the inductive assumption, $P \setminus \{x\}$ can be partitioned into $k$ chains or $k$ antichains. Let us argue that $P \setminus \{x\}$ can actually be partitioned into $k$ chains. When $n-1 = k^2$, this follows from Dilworth's theorem (or its dual version applied to $(P \setminus \{x\})^*$) as $h_k(P \setminus \{x\}) = m_k(\tau_{k,n-1}) = 0$. Otherwise, when $n-1 > k^2$, if $P \setminus \{x\}$ could be partitioned into $k$ antichains, then each of them would have to intersect the chain $L \setminus \{x\}$, which has at least $k$ elements, and one of them would have at least $k+1$ elements, contradicting our assumption above.

Suppose that the $k$ chains decomposing $P \setminus \{x\}$ are $L_1, \ldots, L_k$. It suffices to show that $x \ge \max L_i$ for some $L_i$, since then $L_1, \ldots, L_{i-1}, L_i \cup \{x\}, L_{i+1}, \ldots, L_k$ form a partition of $P$ into $k$ chains. Let $y$ be the second largest element of $L$. Clearly, $y \in L_i$ for some $i$. If $y = \max L_i$, then there is nothing left to prove, so suppose that $z = \max L_i > y$ and consider the set $L \cup \{z\} \subseteq P$. It is easy to see that $y$ belongs to at least $\binom{k+\ell}{k} + \binom{k+\ell-1}{k-1}$ chains of length $k+1$, contradicting our assumption.

\subsection{Posets of height smaller than $k+\ell$}

The case $\height(P) \le k+\ell-1$ can be easily resolved with the use of Corollary~\ref{cor:large-surplus}. Indeed, if $\height(P) \le k+\ell-1$, then
\[
s_k(P) \ge n - k(k+\ell-1) = k+q > k.
\]
Since $\height(P) \ge \width(P)$, Corollary~\ref{cor:large-surplus} implies that $h_k(P) \ge 2^{\sqrt{k/3}-1}$. On the other hand, our assumption that $n \le k^2 + ck^{3/2}/\log_2 k$ and~\eqref{eq:mktau} yield
\begin{equation}
  \label{eq:mk-exp-sqrt-k}
  \begin{split}
    m_k(\tau_{k,n}) & \le k \binom{\lceil n/k \rceil}{k+1} \le k \binom{k+c\sqrt{k}/\log_2 k+1}{k+1} \\
    & \le k \binom{2k}{c\sqrt{k}/\log_2 k} \le \left(\frac{2e\sqrt{k}\log_2k}{c}\right)^{\frac{c \sqrt{k}}{\log_2k}} < 2^{\sqrt{k}/4}.
  \end{split}
\end{equation}
The fourth inequality above is $\binom{a}{b} \le (ea/b)^b$ and the last inequality follows since $c < 1/2$ and $k$ is sufficiently large.

\subsection{Posets of height $k+\ell$}

In view of the above considerations, for the remainder of the proof, we may assume that
\begin{equation}
  \label{eq:height-kpell}
  \height(P) = k + \ell \ge \width(P).
\end{equation}
Since $n \le \height(P)\width(P)$, this means, in particular, that $n > k^2+k$, as otherwise $\ell = 0$ and we have assumed above that $n > k^2$. As $n > k(k+\ell)$, see~\eqref{eq:n-q-k-ell}, assumption~\eqref{eq:height-kpell} implies that $P$ cannot be decomposed into $k$ chains or $k$ antichains. Thus, we shall show that $h_k(P) > m_k(\tau_{k,n})$, unless $n = k^2+k+1$ and $P$ is one of the posets described in Example~\ref{example:extremal-posets}. Observe that
\[
m_k(\tau_{k,n}) = k \binom{k+\ell}{k+1} + q \binom{k+\ell}{k} = \left( q + \frac{k \ell}{k+1} \right) \binom{k+\ell}{k},
\]
cf.~\eqref{eq:mktau} and~\eqref{eq:n-q-k-ell}. In particular, since $\ell, q \le k$ and $k$ is sufficiently large,
\begin{equation}
  \label{eq:mktau-k-to-ell}
  m_k(\tau_{k,n}) \le (q + \ell) \binom{k+\ell}{\ell} \le (k + \ell) \binom{k+\ell}{\ell} < k^{2\ell}.
\end{equation}

\subsubsection{The key lemma}

Let $(A_i)_{i=1}^{k+\ell}$ be the canonical decomposition of $P$ into antichains and recall the definition of $G_i$ from Section~\ref{sec:decomposition}. We shall provide various lower bounds on the number of homogenous sets by analyzing the graphs $G_i$ for various indices $i$ such that $A_i \cup A_{i+1}$ contains an antichain with $k+1$ elements. First and foremost, we shall be looking at $i \in F$, where
\begin{equation}
  \label{eq:F}
  F = \{i \in [k+\ell] \colon |A_i| \ge k+1\}.
\end{equation}
We start by establishing a lower bound on the size of $F$.
\begin{obs}
  \label{obs:F-size}
  For every $I \subseteq [k+\ell]$,
  \[
  |F| \ge \frac{q + \sum_{i \in I} (k-|A_i|)}{\ell}.
  \]
\end{obs}
\begin{proof}
  Recall that the sets $(A_i)_{i = 1}^{k+\ell}$ form a partition of $P$ into antichains and that $|A_i| \le \width(P) \le k+\ell$ for all $i$. Hence,
  \[
  \begin{split}
    n - \sum_{i \in I} |A_i| & = \sum_{i \not\in I \cup F} |A_i| + \sum_{i \in F \setminus I} |A_i| \le (k + \ell - |I \cup F|) \cdot k + |F \setminus I| \cdot (k+\ell) \\
    & = (k + \ell - |I|) \cdot k + |F \setminus I| \cdot \ell \le n - q - |I| \cdot k + |F| \cdot \ell.\qedhere
  \end{split}
\]
\end{proof}

Recall the definitions of $u_i$, $\Sigma_i$, $A_i'$, and $B_i$ from Section~\ref{sec:decomposition}. The following lemma is key
\begin{lemma}
  \label{lemma:F}
  If $i \in F \cap [k+\ell-1]$, then $A_i \cup B_{i+1}$ contains at least $2^{\min\{k, |B_{i+1}|\}}$ antichains with $k+1$ elements and
  \[
  \Sigma_i \ge \Sigma_{i+1} + \sum_{y \in A_{i+1}' \setminus B_{i+1}} u_{i+1}(y) \ge \Sigma_{i+1} + |A_{i+1}'| - |B_{i+1}|.
  \]
\end{lemma}
\begin{proof}
  Note that for any $Y \subseteq B_{i+1}$, the set $Y \cup (A_i \setminus N_{G_i}(Y))$ is an antichain with at least $|A_i|$ elements. Since $|A_i| \ge k+1$, each of these antichains that additionally satisfies $|Y| \le k+1$ contains a $(k+1)$-element subset $L$ such that $L \cap B_{i+1} = Y$. This proves the first assertion of the lemma. The second assertion holds since
\[
\Sigma_i = \sum_{xy \in G_i} u_{i+1}(y) = \sum_{y \in A_{i+1}'} u_{i+1}(y) \deg_{G_i}(y) \ge \sum_{y \in A_{i+1}'} u_{i+1}(y) + \sum_{i \in A_{i+1}' \setminus B_{i+1}} u_{i+1}(y).\qedhere
\]
\end{proof}

\subsubsection{The first lower bound on $\min_i |A_i'|$}

In view of Lemma~\ref{lemma:F}, we shall aim at proving a lower bound on the minimum size of $A_{i+1}'$. We first derive a somewhat weak bound on $\min_i |A_i'|$ from Corollary~\ref{cor:large-surplus}.

\begin{claim}
  Either $h_k(P) > m_k(\tau_{k,n})$ or $|A_i'| \ge k/3$ for all $i \in [k+\ell]$, possibly after substituting $P^*$ for $P$.
\end{claim}
\begin{proof}
  Suppose that $|A_i'| < k/3$ for some $i$. If $\width(P) < \height(P)$, then we let $P' = P \setminus A_i'$ and note that $\height(P') = \height(P) - 1 = k+\ell-1$ as every chain of maximum length in $P$ contains one element of $A_i'$. Thus
\[
s_k(P') = |P'| - k \cdot \height(P') \ge k(k+\ell) - |A_i'| - k(k+\ell-1) \ge 2k/3.
\]
Since $\width(P') \le \width(P) \le \height(P) - 1 = \height(P')$, we may apply Corollary~\ref{cor:large-surplus} with $t = 2k/9$ to $P'$ and conclude that, recalling~\eqref{eq:mk-exp-sqrt-k},
\[
h_k(P) \ge h_k(P') \ge 2^{\sqrt{2k}/3-1} > m_k(\tau_{k,n}).
\]
If $\width(P) = \height(P)$, then we let $(A_i^*)_{i=1}^{k+\ell}$ be the canonical decomposition of $P^*$ into antichains. Now, if $|(A_j^*)'| \ge k/3$ for all $j$, then we work with $P^*$ instead of $P$. Otherwise, if $|(A_j^*)'| < k/3$ for some $j$, then we let $P' = P \setminus (A_i' \cup (A_j^*)')$ and note that $\height(P') < \height(P)$ and $\width(P') < \width(P)$. Consequently, letting $P'' = P'$ or $P'' = (P')^*$ so that $\width(P'') \le \height(P'')$, we have
\[
s_k(P'') = |P''| - k \cdot \height(P'') \ge k(k+\ell) - |A_i'| - |(A_j^*)'| - k(k+\ell-1) \ge k/3.
\]
Now, we may again apply Corollary~\ref{cor:large-surplus} with $t = k/9$ to $P''$ to conclude that, again recalling~\eqref{eq:mk-exp-sqrt-k},
\[
h_k(P) \ge h_k(P'') \ge 2^{\sqrt{k}/3 - 1} > m_k(\tau_{k,n}).\qedhere
\]
\end{proof}

\subsubsection{Posets with large $F$}

For the remainder of the proof, we may and shall assume that $|A_i'| \ge k/3$ for all $i \in [k+\ell]$. Together with Lemma~\ref{lemma:F}, this bound already allows us to deal with the case $|F| \ge 40 \log_2 k$. The~crucial observation here is the following.

\begin{obs}
  \label{obs:chains-per-element}
  If some element of $P$ belongs to more than $2k/\ell$ chains of length $k+\ell$, then $h_k(P) > m_k(\tau_{k,n})$.
\end{obs}
\begin{proof}
  Let $M = 2k/\ell$ and suppose that some $y \in P$ is contained in $M$ different chains of length $k+\ell$. We shall show that this implies that $y$ belongs to more than $\binom{k+\ell}{k}$ chains of length $k+1$, which, by the inductive assumption, implies that $h_k(P) > m_k(\tau_{k,n})$, see~\eqref{eq:mkn-mknm}. This follows easily from Corollary~\ref{cor:signatures}~(\ref{item:signatures-local}) as, letting $m = \log_2 M + 1$, since $\ell, m \ll \sqrt{k}$, we have that
  \[
  \exp\left(-\frac{2 (\ell-1) m}{k}\right) M \binom{k+\ell}{k+1} \ge \frac{3k}{2\ell} \binom{k+\ell}{k+1} = \frac{3k}{2(k+1)} \binom{k+\ell}{k}.\qedhere
  \]
\end{proof}

\begin{claim}
  \label{claim:F-size}
  If $|F| \ge 40\log_2 k$, then $h_k(P) > m_k(\tau_{k,n})$.
\end{claim}
\begin{proof}
  We may assume that $|B_{i+1}| < 2\ell\log_2k$ for every $i \in F$ as otherwise Lemma~\ref{lemma:F} implies that $h_k(P) \ge k^{2\ell} > m_k(\tau_{k,n})$. We have also assumed that $|A_i'| \ge k/3$ for every $k$ and hence, again by Lemma~\ref{lemma:F},
  \[
  \Sigma_{i+1} - \Sigma_i \ge k/4 \quad \text{ for every $i \in F \cap [k+\ell-1]$}.
  \]
  Partition $F \cap [k+\ell-1]$ into $I_1$ and $I_2$ with $|I_1| \ge 32 \log_2 k$ and $|I_2| \ge 4 \log_2 k$ such that $\min I_1 > \max I_2$. By~\eqref{eq:sum-ui-monotone},
  \[
  \Sigma_{\min I_1} = \Sigma_{k+\ell} + \sum_{i = \min I_1}^{k+\ell-1} (\Sigma_i - \Sigma_{i+1}) \ge |I_1| \cdot k/4 \ge 8 k \log_2 k.
  \]
  Now, consider an arbitrary $i \in I_2$. In accordance with Observation~\ref{obs:chains-per-element}, we may assume that $u_{i+1}(y) \le 2k/\ell$ for every $y \in A_{i+1}$. As $i+1 \le \min I_1$ and $|B_{i+1}| \le 2\ell\log_2k$,
  \[
  \sum_{y \in B_{i+1}} u_{i+1}(y) \le |B_{i+1}| \cdot 2k/\ell \le 4k \log_2 k \le \Sigma_{\min I_1} / 2 \le \Sigma_{i+1}/2,
  \]
  where the final inequality follows from~\eqref{eq:sum-ui-monotone}. Hence, by Lemma~\ref{lemma:F},
  \[
  \Sigma_i \ge \Sigma_{i+1} + \sum_{y \in A_{i+1}' \setminus B_{i+1}} u_{i+1}(y) \ge 3\Sigma_{i+1}/2.
  \]
  It follows that
  \[
  \Sigma_1 \ge (3/2)^{|I_2|} \cdot \Sigma_{\min I_1} \ge k^3,
  \]
  implying that there is an $x \in A_1$ which belongs to more than $2 k / \ell$ (actually more than $k^2/2$) chains of length $k+\ell$. Consequently Observation~\ref{obs:chains-per-element} implies that $h_k(P) > m_k(\tau_{k,n})$.
\end{proof}

\subsubsection{A sufficient condition on $\Sigma_1$}

For the remainder of the proof, we shall therefore assume that $|F| \le 40\log_2k$ and, as a consequence of Observation~\ref{obs:F-size}, that $q \le 40\ell\log_2k$. In view of Corollary~\ref{cor:signatures}~(\ref{item:signatures-global}), in order to conclude that $h_k(P) > m_k(\tau_{k,n})$, it is enough to provide a sufficiently strong lower bound on $\Sigma_1$, the number of chains of length $k+\ell$ in $P$. To this end, define
\begin{equation}
  \label{eq:S}
  S = \left(1 + \frac{q}{\ell}\right)k + 50\sqrt{k}\log_2k
\end{equation}
and note that our assumption on $q$ implies that $S \le 42k\log_2k$.

\begin{obs}
  \label{obs:Sigma-1}
  If $\Sigma_1 \ge S$, then $h_k(P) > m_k(\tau_{k,n})$.
\end{obs}
\begin{proof}
  Since we have assumed that $q \le 40\ell\log_2k$, then
  \[
  S \ge \left(1+\frac{1}{\sqrt{k}}\right) \left(1+\frac{q}{\ell}\right)(k+1).
  \]
  As $\binom{k+\ell}{k+1} = \frac{\ell}{k+1}\binom{k+\ell}{k}$, it now follows from Corollary~\ref{cor:signatures}~(\ref{item:signatures-global}) that
  \[
  \begin{split}
    h_k(P) & \ge \exp\left(-\frac{3\ell\log_2k}{k}\right)\left(1+\frac{1}{\sqrt{k}}\right)\left(1+\frac{q}{\ell}\right)(k+1) \binom{k+\ell}{k+1} \\
    & \ge \left(1 - \frac{1}{2\sqrt{k}}\right)\left(1 + \frac{1}{\sqrt{k}}\right) (\ell + q) \binom{k+\ell}{k} > (\ell + q) \binom{k+\ell}{k} \ge m_k(\tau_{k,n}),
  \end{split}
  \]
  where the second inequality holds since $\ell \le \sqrt{k} / (300\log_2 k)$ and the last inequality is~\eqref{eq:mktau-k-to-ell}.
\end{proof}

\subsubsection{The second lower bound on $\min_i |A_i'|$}

We shall now focus on proving the following strong lower bound on $\min_i |A_i'|$.

\begin{claim}
  \label{claim:Ap-2nd}
  Either $h_k(P) > m_k(\tau_{k,n})$ or $|A_i'| \ge k - 240 \ell \log_2 k$ for all $i \in [k+\ell]$.
\end{claim}

Recall from Section~\ref{sec:decomposition} that for each $i \in [k+\ell-1]$, the set $(A_i \setminus A_i') \cup A_{i+1}'$ is an antichain and consequently,
\begin{equation}
  \label{eq:Ai-Aip-Aipp}
  |A_i| - |A_i'| + |A_{i+1}'| \le \width(P) \le k+\ell.
\end{equation}
With foresight, define
\begin{equation}
  \label{eq:Fp}
  F' = \{i \in [k+\ell-1] \colon |A_i| - |A_i'| + |A_{i+1}'| \ge k+1\}.
\end{equation}
Let $J$ be an arbitrary subset of $[k+\ell]$. By~\eqref{eq:Ai-Aip-Aipp} and~\eqref{eq:Fp},
\begin{equation}
  \label{eq:Ap-telescope}
  \sum_{i \in J} (|A_{i+1}'| - |A_i'|) \le \sum_{i \in J \cap F'} (k+\ell-|A_i|) + \sum_{j \in J \setminus F'} (k - |A_i|) \le \sum_{i \in J} (k-|A_i|) + |F'| \cdot \ell.
\end{equation}
Now, fix a $j \in [k+\ell]$ and let $J = \{j, \ldots, k+\ell-1\}$. It follows from \eqref{eq:Ap-telescope} and Observation~\ref{obs:F-size} with $I = J \cup \{k+\ell\}$ that (recalling that $A_{k+\ell}' = A_{k+\ell}$)
\[
k - |A_j'| = k - |A_{k+\ell}'| + \sum_{i \in J} (|A_{i+1}'| - |A_i'|) \le \sum_{i \in J \cup \{k+\ell\}} (k - |A_i|) + |F'| \cdot \ell \le (|F| + |F'|) \cdot \ell.
\]
Therefore, in order to establish Claim~\ref{claim:Ap-2nd}, it suffices to prove that $|F'| > 200\log_2 k$ implies that $h_k(P) > m_k(\tau_{k,n})$. This fact is a fairly straightforward consequence of the following lemma, which one may consider as the `dual' version of Lemma~\ref{lemma:F}.

\begin{lemma}
  \label{lemma:Fp}
  Either $h_k(P) > m_k(\tau_{k,n})$ or $\Sigma_i \ge \Sigma_{i+1} + k/4$ for all $i \in F'$.
\end{lemma}
\begin{proof}
  Fix some $i \in F'$ and let
  \[
  C_i = \{x \in A_i' \colon \deg_{G_i}(x, A_{i+1}') \le 1\}.
  \]
  Observe that for every $X \subseteq C_i$, the set $(A_i \setminus A_i') \cup X \cup (A_{i+1}' \setminus N_{G_i}(X))$ is an antichain with at least $|A_i| - |A_i'| + |A_{i+1}'|$ elements. Each of these antichains that additionally satisfies $|X| \le k+1$ contains a $(k+1)$-element subset $L$ such that $L \cap A_i' = X$. Therefore, if $|C_i| \ge 2\ell\log_2k$, then $h_k(P) \ge k^{2\ell} > m_k(\tau_{k,n})$. On the other hand, if $|C_i| < 2\ell\log_2k$, then
  \[
  \begin{split}
    \Sigma_i & = \sum_{xy \in G_i} u_{i+1}(y) = \sum_{y \in A_{i+1}'} u_{i+1}(y)\deg_{G_i}(y) \ge \Sigma_{i+1} + e_{G_i}(A_i', A_{i+1}') - |A_{i+1}'| \\
    & \ge \Sigma_{i+1} + 2|A_i'| - |C_i| - |A_{i+1}'| \ge \Sigma_{i+1} + |A_i'| - (|A_{i+1}'| - |A_i'|) - 2\ell\log_2k.
  \end{split}
  \]
  By Observation~\ref{obs:F-size} with $I = \{i\}$, we have $k - |A_i| \le |F| \ell$. Therefore, by~\eqref{eq:Ai-Aip-Aipp} and our assumption on $|F|$,
  \[
  |A_{i+1}'| - |A_i'| \le k+\ell-|A_i| \le (|F| + 1) \cdot \ell \ll k.
  \]
  Consequently, as we have assumed that $|A_i'| \ge k/3$, we have $\Sigma_i \ge \Sigma_{i+1} + k/4$.
\end{proof}

\begin{proof}[Proof of Claim~\ref{claim:Ap-2nd}]
  Observe first that Lemma~\ref{lemma:Fp} yields $\Sigma_1 \ge |F'| \cdot k/4$. Indeed, similarly as in the proof of Claim~\ref{claim:F-size},
  \[
  \Sigma_1 = \Sigma_{k+\ell} + \sum_{i=1}^{k+\ell-1} (\Sigma_i - \Sigma_{i+1}) \ge |F'| \cdot k/4.
  \]
  As $q \le 40\ell\log_2k$ by our assumption, if $|F'| > 200\log_2k$, then $\Sigma_1 \ge 50k\log_2k \ge S$, see~\eqref{eq:S}, and consequently $h_k(P) > m_k(\tau_{k,n})$ by Observation~\ref{obs:Sigma-1}.
\end{proof}

\subsubsection{Strengthening Lemmas~\ref{lemma:F} and~\ref{lemma:Fp}}

For the remainder of the proof, we shall assume that $|A_i'| \ge k-240\ell\log_2k$ for every $i$. This assumption will allow us to prove the following strengthening of Lemmas~\ref{lemma:F} and~\ref{lemma:Fp}. Recall the definitions of $F$ from~\eqref{eq:F} and $F'$ from~\eqref{eq:Fp}.

\begin{lemma}
  \label{lemma:Fp-strong}
  Either $h_k(P) > m_k(\tau_{k,n})$ or
  \begin{enumerate}[(i)]
  \item
    \label{item:Fp-strong-one}
    $\Sigma_i \ge \Sigma_{i+1} + k - \sqrt{k}$ for all $i \in F \cup F'$ and
  \item
    \label{item:Fp-strong-two}
    $\Sigma_i \ge \Sigma_{i+1} + 2k - 3\sqrt{k}$ for all $i \in F$ such that $i < \max F'$.
  \end{enumerate}
\end{lemma}
\begin{proof}
  Suppose first that $i \in F$. We may assume that $|B_{i+1}| \le 2\ell\log_2k$, as otherwise $h_k(P) \ge k^{2\ell} > m_k(\tau_{k,n})$ by Lemma~\ref{lemma:F}. Consequently,
  \[
  \Sigma_i - \Sigma_{i+1} \ge |A_{i+1}'| - |B_{i+1}| \ge k - 242\ell\log_2k \ge k - \sqrt{k},
  \]
  as $\ell \le \sqrt{k}/(300 \log_2 k)$ and we have assumed that $|A_{i+1}'| \ge k - 240\ell\log_2k$.

  Suppose now that $i \in F'$. As in the proof of Lemma~\ref{lemma:Fp}, let
  \[
  C_i = \{x \in A_i' \colon \deg_{G_i}(x, A_{i+1}') \le 1\}
  \]
  and recall that either $|C_i| < 2\ell\log_2k$ or $h_k(P) \ge k^{2\ell} > m_k(\tau_{k,n})$ and that
  \begin{equation}
    \label{eq:Fp-strong}
    \Sigma_i \ge \Sigma_{i+1} + |A_i'| - (|F|+1) \cdot \ell - |C_i| \ge \Sigma_{i+1} +k - \sqrt{k},
  \end{equation}
  as $\ell \le \sqrt{k}/(300 \log_2 k)$ and we have assumed that $|F| \le 40\log_2k$ and that $|A_i'| \ge k - 240\ell\log_2k$.

  We now turn to proving~(\ref{item:Fp-strong-two}). First, for each $i \in [k+\ell]$, define
  \[
  A_i'' = \{x \in A_i \colon u_i(x) \ge 2\}
  \]
  and observe that $A_i'' \supseteq A_i' \setminus C_i$. Indeed, if $x \in A_i' \setminus C_i$, then by~\eqref{eq:ui-uip},
  \[
  u_i(x) = \sum_{xy \in G_i} u_{i+1}(y) \ge \deg_{G_i}(x,A_{i+1}') \ge 2.
  \]
  Consequently, if $i \in F'$ and $h_k(P) \le m_k(\tau_{k,n})$, then $|A_i''| \ge |A_i'| - |C_i| \ge k - \sqrt{k}$, see~\eqref{eq:Fp-strong}.
  
  Assume now that $F' \neq \emptyset$ and let $f' = \max F'$. We claim that either $h_k(P) > m_k(\tau_{k,n})$ or $|A_{i+1}''| \ge k - 2\sqrt{k}$ for each $i < f'$. To this end, observe first that for each $i \in [k+\ell-1]$, the set $(A_i \setminus A_i'') \cup A_{i+1}''$ is an antichain and therefore $|A_i| - |A_i''| + |A_{i+1}''| \le k + \ell$. Indeed, if $x \in A_i$ and $y \in A_{i+1}$ satisfy $u_i(x) < u_{i+1}(y)$, then~\eqref{eq:ui-uip} implies that $xy \not\in G_i$. With foresight, let
  \[
  F'' = \{i \in [f' - 1] \colon |A_{i+1}''| - |A_i''| + |A_i| \ge k+1\}.
  \]
  Assume that $\min_{i \le f'} |A_i''| < k - 2\sqrt{k}$ and let $i \in [f']$ be the largest index such that $|A_i''| < k - 2\sqrt{k}$. Since, as we have shown above, $|A_{f'}''| \ge k - \sqrt{k}$, then
  \begin{equation}
    \label{eq:Fpp}
    \begin{split}
      \sqrt{k} & < |A_{f'}''| - |A_i''| = \sum_{j = i}^{f'-1} (|A_{i+1}''| - |A_i''|) \le \sum_{j = i}^{f'-1} (k-|A_i|) + |F'' \cap \{i, \ldots, f'-1\}| \cdot \ell \\
      & \le |F| \cdot \ell + |F'' \cap \{i+1, \ldots, f'-1\}| \cdot \ell + \ell,
    \end{split}
  \end{equation}
  where the last inequality follows from Observation~\ref{obs:F-size}. Since $\ell \le \sqrt{k}/(300 \log_2 k)$ and we have assumed that $|F| \le 40\log_2k$, it follows from~\eqref{eq:Fpp} that $|F'' \cap \{i+1, \ldots, f'-1\}| \ge 50 \log_2k$. We claim that this implies that $h_k(P) > m_k(\tau_{k,n})$. To this end, consider some $j \in F''$, let
  \[
  C_j' = \{x \in A_j'' \colon \deg(x, A_{j+1}'') \le 1\},
  \]
  and note that for every $X \subseteq C_j'$, the set $(A_j \setminus A_j'') \cup X \cup (A_{j+1}'' \setminus N_{G_j}(X))$ is an antichain with at least $k+1$ elements. Therefore, if $|C_j''| \ge 2\ell\log_2k$, then $h_k(P) \ge k^{2\ell} > m_k(\tau_{k,n})$. On the other hand, if $|C_j| < 2\ell\log_2k$, then
  \[
  \Sigma_j - \Sigma_{j+1} \ge e_{G_j}(A_j'' + A_{j+1}'') - |A_{j+1}''| \ge 2|A_j''| - (k+\ell) - 2\ell\log_2k.
  \]
  Therefore, if $j > i$ and $j \in F''$, then $\Sigma_j - \Sigma_{j+1} \ge k - o(k)$ and consequently, recalling~\eqref{eq:S},
  \[
  \Sigma_1 \ge |F'' \cap \{i+1, \ldots, f'-1\}| \cdot (k-o(k)) \ge (1-o(1)) \cdot 50k\log_2k \ge S,
  \]
  which, by Observation~\ref{obs:Sigma-1}, yields $h_k(P) > m_k(\tau_{k,n})$.

  Finally, suppose that $i \in F$ and $i < \max F'$. We may now assume that $|A_{i+1}''| \ge k - 2\sqrt{k}$ and therefore, by Lemma~\ref{lemma:F},
  \[
  \Sigma_i - \Sigma_{i+1} \ge \sum_{y \in A_{i+1}' \setminus B_{i+1}} u_{i+1}(y) \ge |A_{i+1}''| + |A_{i+1}'| - |B_{i+1}| \ge 2k - 3\sqrt{k}.\qedhere
  \]
\end{proof}

\subsubsection{Narrowing down to almost extremal posets}

The following observation will further narrow down our search for $P$ with $h_k(P) \le m_k(\tau_{k,n})$.

\begin{obs}
  \label{obs:FFp}
  If $|(F \cup F') \cap [k+\ell-1]| \ge \frac{q+1}{\ell}$, then $h_k(P) > m_k(\tau_{k,n})$.
\end{obs}
\begin{proof}
  Recall that $\Sigma_{k+\ell} = |A_{k+\ell}| = |A_{k+\ell}'| \ge k - 240\ell\log_2k \ge k - \sqrt{k}$. By Lemma~\ref{lemma:Fp-strong}~(\ref{item:Fp-strong-one}) and~\eqref{eq:sum-ui-monotone},
  \[
  \Sigma_1 \ge \Sigma_{k+\ell} + \sum_{i \in (F \cup F') \cap [k+\ell-1]} (\Sigma_i - \Sigma_{i+1}) \ge \left(1 + \frac{q+1}{\ell}\right)\left(k-\sqrt{k}\right).
  \]
  Now, since $q \le 40\ell\log_2k$ and $\ell \le \sqrt{k}/(300 \log_2 k)$, then $\Sigma_1 \ge S$ and the conclusion follows from Observation~\ref{obs:Sigma-1}.
\end{proof}

Recall the definition of $F$ from~\eqref{eq:F}. We shall now split into cases depending on whether or not $k + \ell \in F$, that is, whether or not $|A_{k+\ell}| \ge k+1$.

\begin{case}
  \label{case:k-ell-not-F}
  $k + \ell \notin F$.
\end{case}

The assumption that $k+\ell \not\in F$ and Observation~\ref{obs:F-size} imply that
\[
|F \cap [k+\ell-1]| = |F| \ge \left\lceil \frac{q + |\{i \in [k+\ell] \colon |A_i| < k\}|}{\ell} \right\rceil.
\]
By Observation~\ref{obs:FFp}, we may assume that $|A_i| \ge k$ for all $i$, $\ell$ divides $q$, and $|F| = q/\ell$, as otherwise $h_k(P) > m_k(\tau_{k,n})$. This implies that $|A_i| = k+\ell$ for each $i \in F$ and $|A_i| = k$ otherwise.

\begin{subcase}
  $F' \neq \emptyset$ and $\max F' > \min F$.
\end{subcase}

Let $j = \max F'$ and let $i$ be the largest element of $F$ that is smaller than $j$. By Lemma~\ref{lemma:Fp-strong}~(\ref{item:Fp-strong-two}),
\[
\Sigma_i - \Sigma_{i+1} \ge 2k - 3\sqrt{k}.
\]
Consequently, using Lemma~\ref{lemma:Fp-strong} as in the proof of Observation~\ref{obs:FFp},
\[
\Sigma_1 \ge \Sigma_{k+\ell} + \sum_{j \in F \setminus \{i\}} (\Sigma_j - \Sigma_{j+1}) + \Sigma_i - \Sigma_{i+1} \ge \left(1 + \frac{q}{\ell}\right)\left(k - \sqrt{k}\right) + k - 2\sqrt{k} \ge S,
\]
which yields $h_k(P) > m_k(\tau_{k,n})$.

\begin{subcase}
  $F' = \emptyset$ or $\max F' \le \min F$.
\end{subcase}

If $F' \neq \emptyset$, then we may also assume that $\min F' \ge \min F$. Indeed, otherwise
\[
|(F \cup F') \cap [k+\ell-1]| \ge |F| + 1 \ge \frac{q}{\ell}+1
\]
and hence $h_k(P) > m_k(\tau_{k,n})$ by Observation~\ref{obs:FFp}. Hence $\min F' = \max F' = \min F$.

We first claim that $|A_i'| \ge k$ for every $i$. Suppose not and let $i$ be the largest index for which $|A_i'| < k$ and note that $i < k+\ell$ as $A_{k+\ell}' = A_{k+\ell}$ and we have assumed that $|A_j| \ge k$ for each $j \in [k+\ell]$. Consequently,
\[
|A_i| + |A_{i+1}'| - |A_i'| \ge |A_i| + k - |A_i'| > |A_i| \ge k,
\]
implying that $i \in F'$ and thus $i = \min F$. But this is impossible as $|A_i| = k+\ell$ and hence $(A_i \setminus A_i') \cup A_{i+1}'$ would be an antichain with more than $k+\ell$ elements.

We now claim that $\max F \le \min F + 1$, that is, that $|F|=1$ or $F = \{f, f+1\}$ for some $f \in [k+\ell-1]$. To see this, note that if $\max F \in F'$, then our assumption implies that $\max F = \min F$, that is, $|F| = 1$. Otherwise, if $\max F \not\in F'$, then
\[
|A_{\max F}| - |A_{\max F}'| + |A_{\max F+1}'| \le k
\]
and hence, as $|A_{\max F+1}'| \ge k$, we have $|A_{\max F}'| = |A_{\max F}| = k+\ell$. Consequently, either $\max F = 1$ or $\max F - 1 \in F'$ and therefore $\max F - 1 = \min F$ by our assumption.

Finally, let $f = \min F$. If $f > 1$ and $|A_f'| > k$, then $f - 1 \in F'$, contradicting our assumption. We may thus assume that $f = 1$ or $|A_f'| = k$. Consequently, if $|F| > 1$, then $F = \{1, 2\}$. Indeed, if $|F| > 1$, then $F = \{f, f+1\}$ and $|A_{f+1}'| = k+\ell$. Moreover, since $A_{f+1}' \cup (A_f \setminus A_f')$ is an antichain and $\width(P) \le k+\ell$, then $|A_f'| = |A_f| = k+\ell$. Finally, we have shown above that if $f > 1$, then $|A_f'| = k$.

\begin{case}
  $k + \ell \in F$.
\end{case}

Consider the poset $\hat{P}$ obtained from $P$ by reversing the $\le$ relation, that is, by letting $x \le y$ in $\hat{P}$ if and only if $y \le x$ in $P$. Clearly, the same sets form chains and antichains in both $P$ and $\hat{P}$. Let $(\hat{A}_i)_{i=1}^{k+\ell}$ be the canonical decomposition of $\hat{P}$ into antichains and observe that $\hat{A}_{k+\ell} = A_1'$. Indeed,
\[
\begin{split}
  \hat{A}_{k+\ell} & = \{x \in \hat{P} \colon \text{there is a chain $L \subseteq \hat{P}$ of length $k+\ell$ with $x = \max L$}\} \\
  & = \{x \in P \colon \text{there is a chain $L \subseteq P$ of length $k+\ell$ with $x = \min L$}\} = A_1'.
\end{split}
\]
Thus, we may assume that $|A_1'| > k$ as otherwise $\hat{P}$ falls into Case~\ref{case:k-ell-not-F}. Thus $1 \in F$. We show that this implies that $\Sigma_1 \ge S$ and consequently, by Observation~\ref{obs:Sigma-1}, that $h_k(P) > m_k(\tau_{k,n})$.

Since $|A_{k+\ell}'| = |A_{k+\ell}| > k$, then $k+\ell-1 \in F'$ and hence, by Lemma~\ref{lemma:Fp-strong}, we may assume that $\Sigma_{k+\ell-1} \ge \Sigma_{k+\ell} + k - \sqrt{k} \ge 2k-\sqrt{k}$ and $\Sigma_i \ge \Sigma_{i+1} + 2k-3\sqrt{k}$ for all $i \in F \cap [k+\ell-2]$. This yields
\begin{equation}
  \label{eq:Sigma-1-case-2}
  \Sigma_1 \ge \big( 2 + 2|F \cap [k+\ell-2]| + |(F' \setminus F) \cap [k+\ell-2]|\big) (k - 2\sqrt{k}).
\end{equation}
By our assumption that $1 \in F$ and Observation~\ref{obs:F-size},
\begin{equation}
  \label{eq:F-case-2}
  |F \cap [k+\ell-2]| \ge \max\{1, \lceil q / \ell \rceil - 2\}.
\end{equation}
It follows that either $q = 3\ell$ and we have equality in~\eqref{eq:F-case-2} or
\[
\Sigma_1 \ge \left( 1 + \frac{q+1}{\ell} \right)(k - 2\sqrt{k}) = \left(1 + \frac{q}{\ell}\right)k + \frac{k}{\ell} - 2\sqrt{k}\left(1 + \frac{q+1}{\ell}\right) \ge S;
\]
to see the last inequality, recall that $\ell \le \sqrt{k}/(300 \log_2k)$. The former (i.e., $q = 3\ell$ and equality in~\eqref{eq:F-case-2}) implies that $F = \{1, k+\ell-1, k+\ell\}$, $A_i = k+\ell$ for all $i \in F$, and $|A_i| = k$ for all $i \not\in F$. Since $A_{k+\ell}' \cup (A_{k+\ell-1} \setminus A_{k+\ell-1}')$ is an antichain and $\width(P) \le k+\ell = |A_{k+\ell}'|$, we must have $|A_{k+\ell-1}'| = |A_{k+\ell-1}| = k+\ell$ and consequently, $k+\ell-2 \in F'$. Now, \eqref{eq:Sigma-1-case-2} again yields $\Sigma_1 \ge 5(k-2\sqrt{k}) \ge S$.

\subsection{Almost extremal posets}

Summarizing the above discussion, if $h_k(P) \le m_k(\tau_{k,n})$, then $\height(P) = k+\ell$ and either $P$ or $\hat{P}$ (the poset obtained from $P$ by reversing the $\le$ relation) satisfy one of the following two lists of conditions:
\begin{enumerate}[(1)]
\item
  \label{item:extremal-1}
  $F = \{1, 2\}$, $|A_1'| = |A_2'| = k+\ell$, and $|A_i'| = k$ for every $i \ge 3$,
\item
  \label{item:extremal-2}
  $F = \{f\}$ for some $f \in [k+\ell-1]$ and $|A_i'| \ge k$ for all $i$; if $f > 1$, then $|A_f'| = k$.
\end{enumerate}
From now on, we shall have to count homogenous sets somewhat more carefully, as there are posets of either of these two types that contain fewer than $m_k(\tau_{k,n})$ chains of length $k+1$ and fewer than $m_k(\tau_{k,n})$ antichains with $k+1$ elements. (So far, we have always managed to show that our poset contains more than $m_k(\tau_{k,n})$ homogenous sets of one of the two types.)

\subsubsection{Bounding the number of antichains}

We first derive a lower bound for the number of $(k+1)$-element antichains which we shall use in both~(\ref{item:extremal-1}) and~(\ref{item:extremal-2}). To this end, for each $i \in [k+\ell-1]$, let
\[
D_{i+1} = \{x \in A_{i+1}' \colon \deg_{G_i}(x) = 2\} \subseteq A_{i+1}' \setminus B_{i+1}.
\]
Note for future reference that it follows from (the proof of) Lemma~\ref{lemma:F} that for each $i \in F$,
\begin{equation}
  \label{eq:lemma-F}
  \Sigma_i - \Sigma_{i+1} = \sum_{y \in A_{i+1}'} u_{i+1}(y)(\deg_{G_i}(y)-1) \ge 2|A_{i+1}'| - |D_{i+1}| - 2|B_{i+1}|,
\end{equation}
which improves the lower bound of $|A_{i+1}'| - |B_{i+1}|$ for $\Sigma_i - \Sigma_{i+1}$ stated in Lemma~\ref{lemma:F} whenever $D_{i+1}$ is smaller than $A_{i+1}' \setminus B_{i+1}$. On the other hand, if $D_{i+1}$ is large, then there are many $(k+1)$-element antichains in $A_i \cup D_{i+1}$, as the following lemma shows.

\begin{lemma}
  \label{lemma:C-antichains}
  Let $i \in [k+\ell-1]$ and suppose that $|A_i| = k+\ell$ and $|A_{i+1}'| = k$. If $|D_{i+1}| \ge k-k^{2/3}$, then $A_i \cup D_{i+1}$ contains at least $20(\ell-1)\binom{k+\ell}{k+1}$ antichains with $k+1$ elements that intersect $D_{i+1}$.
\end{lemma}
\begin{proof}
  For each $z \in A_i$, let $H_z$ be an arbitrary tree with vertex set $N_{G_i}(z) \cap D_{i+1}$ and let $H$ be the multigraph with vertex set $D_{i+1}$ which is the union of all $H_z$ as $z$ ranges over $A_i$. Clearly,
  \[
  e(H) \ge \sum_{z \in A_i} (\deg_{G_i}(z, D_{i+1}) - 1) \ge 2|D_{i+1}| - |A_i| \ge k-3k^{2/3},
  \]
  as $|A_i| = k+\ell$ and $\ell \ll k^{2/3}$. Note crucially that:
  \begin{enumerate}[(i)]
  \item
    \label{item:HX}
    For every $X \subseteq D_{i+1}$, the set $X \cup (A_i \setminus N_{G_i}(X))$ is an antichain.
  \item
    \label{item:eHX}
    For every $X \subseteq D_{i+1}$, we have $|N_{G_i}(X)| \le 2|X| - e_H(X)$.
  \end{enumerate}
  To see~(\ref{item:eHX}), recall that each $H_z$ is a tree and therefore
  \[
  |N_{G_i}(X)| = 2|X| - \sum_{z \in A_i} \max\{\deg_{G_i}(z,X)-1, 0\} \le 2|X| - \sum_{z \in A_i} e_{H_z}(X) = 2|X| - e_H(X).
  \]

  We first show that we may assume that $H$ contains fewer than $\sqrt{k}$ cycles (we consider two parallel edges to be a cycle). Since $\width(P) \le k+\ell \le |A_i|$, (\ref{item:HX}) implies that $|N_{G_i}(X)| \ge |X|$ for each $X \subseteq D_{i+1}$ and hence each component $T$ of $H$ has at most one cycle as otherwise $|N_{G_i}(T)| < |T|$ by~(\ref{item:eHX}). Suppose now that $X_1, \ldots, X_m$ are cycles in $H$. They belong to different components of $H$, so in particular they are vertex-disjoint. By~(\ref{item:eHX}), for any $J \subseteq [m]$, the set $X = \bigcup_{j \in J} X_j$ satisfies $|N_{G_i}(X)| = |X|$. By~(\ref{item:HX}), $A_i \cup D_{i+1}$ contains at least $2^m$ antichains with $k+\ell$ elements and thus, by Lemma~\ref{lemma:KK}, at least $2^{m-1}$ antichains with $k+1$ elements. Finally, as $\ell < \sqrt{k}/(300 \log_2 k)$, then $2^{\sqrt{k}} \gg \ell \binom{k+\ell}{k+1}$, cf.~\eqref{eq:mk-exp-sqrt-k}.

  Assume that $H$ has fewer than $\sqrt{k}$ cycles and delete from $H$ one edge in each of its cycles to obtain a forest $H'$ with $e(H') \ge e(H) - \sqrt{k} \ge k - 4k^{2/3}$. Let $N'$ be the number of sets $X \subseteq D_{i+1}$ with $|X| \le 41$ which are connected in $H'$. As $H' \subseteq H$, it follows from~(\ref{item:eHX}) that $|N_{G_i}(X)| \le |X| + 1$ for each such $X$. Thus by~(\ref{item:HX}), the number $N$ of antichains in $A_i \cup D_{i+1}$ satisfies
  \[
  \begin{split}
    N & \ge \sum_{X \subseteq D_{i+1}} \binom{k+\ell-|N_{G_i}(X)|}{k+1-|X|} \ge N' \binom{k+\ell-42}{k-40} \\
    & \ge N' \cdot \left(\frac{k-40}{k+\ell}\right)^{41} \cdot \binom{k+\ell-1}{k+1} \ge \frac{N'}{2} \cdot\frac{\ell-1}{k} \binom{k+\ell}{k+1},
  \end{split}
  \]
  where we have used the assumption that $\ell = o(k)$. Finally, let $t_1, \ldots, t_m$ be the orders of the trees constituting $H'$. Since $m = k - e(H') \le 4k^{2/3}$, it follows from Lemma~\ref{lemma:connected-sets} that
  \[
  N' \ge \sum_{c=1}^{41} \sum_{j=1}^m (t_j-c) \ge \sum_{c=1}^{41} (k - cm) \ge 41k - O(k^{2/3}) \ge 40k.\qedhere
  \]
\end{proof}

\subsubsection{Almost extremal posets of type~(\ref{item:extremal-1})}

Recall that $P$ is of type~(\ref{item:extremal-1}) if and only if $|A_i| = |A_i'| = k+\ell$ for $i \in \{1, 2\}$ and $|A_i| = |A_i'| = k$ otherwise. In particular, $q = 2\ell$ and hence $S = 3k + 50\sqrt{k}\log_2k$. We first show that $h_k(P) > m_k(\tau_{k,n})$ for each such $P$. Let $b_2 = |B_2|$ and $b_3 = |B_3|$. By Lemma~\ref{lemma:F} and \eqref{eq:lemma-F},
\[
\Sigma_1 \ge \Sigma_3 + 2|A_3'| - |D_3| - 2b_3 + |A_2'| - b_2 \ge 4k + \ell - |D_3| - b_2 - 2b_3.
\]
We may assume that $b_2, b_3 \le 2\ell\log_2k$ as otherwise Lemma~\ref{lemma:F} implies that $h_k(P) > m_k(\tau_{k,n})$. If $|D_3| < k - k^{2/3}$, then $\Sigma_1 \ge S$ and, by Observation~\ref{obs:Sigma-1}, $h_k(P) > m_k(\tau_{k,n})$. Thus, we may assume that $|D_3| \ge k - k^{2/3}$.

Let us carefully count homogenous $(k+1)$-sets in $P$. First, consider the collection of all chains of length $k+1$ obtained by taking a triple $(x_1,x_2,x_3) \in A_1' \times A_2' \times A_3'$ with $x_1 \le x_2 \le x_3$ and an arbitrary set of $k-2$ elements from some chain of length $k+\ell$ that contains $\{x_1,x_2,x_3\}$. The number of such chains containing a fixed triple is
\[
\binom{k+\ell-3}{k-2} = \frac{(k+1)k(k-1)}{(k+\ell)(k+\ell-1)(k+\ell-2)} \binom{k+\ell}{k+1} \ge \left(1 - \frac{\ell-1}{k-1}\right)^3 \binom{k+\ell}{k+1}
\]
and all chains constructed in this way are distinct. Let $N_t$ be the number of such triples. Since neither $G_1$ nor $G_2$ contain any isolated vertices, $A_2 = A_2'$, $A_3 = A_3'$, $e(G_1) \ge 2|A_2'| - b_2$, and $e(G_2) \ge 2|A_3'| - b_3$, then
\[
N_t = \sum_{x \in A_2'} \deg_{G_1}(x) \deg_{G_2}(x) \ge \sum_{x \in A_2'} (\deg_{G_1}(x) + \deg_{G_2}(x) - 1) \ge 2|A_3'| + |A_2'| - b_2 -b_3.
\]
Therefore, the number $N_c$ of chains of length $k+1$ satisfies
\[
N_c \ge (2|A_3'| + |A_2'| - b_2 - b_3) \binom{k+\ell-3}{k+1} \ge \left(1 - \frac{\ell-1}{k-1}\right)^3 (3k+\ell-b_2-b_3) \binom{k+\ell}{k+1}.
\]
To estimate the number of antichains, note that for any $i \in \{1, 2\}$ and any $Y \subseteq B_{i+1}$, the set $Y \cup (A_i \setminus N_{G_i}(Y))$ is an antichain with $k+\ell$ elements. Hence, the number $N_{a,1}$ of $(k+1)$-element antichains that are contained in either $A_1 \cup B_2$ or $A_2 \cup B_3$ satisfies
\[
\begin{split}
  N_a & \ge \sum_{i=2}^3 \sum_{Y \subseteq B_i} \binom{k+\ell-|Y|}{k+1-|Y|} \ge \sum_{i=2}^3 2^{b_i} \binom{k+\ell-b_i}{k+1-b_i} \ge \sum_{i=2}^3 2^{b_i} \left(1 - \frac{\ell-1}{k+\ell-b_i} \right)^{b_i}  \binom{k+\ell}{k+1} \\
  & \ge \left[\left(\frac{7}{4}\right)^{b_2} + \left(\frac{7}{4}\right)^{b_3}\right] \binom{k+\ell}{k+1},
\end{split}
\]
where the last inequality follows since $\ell, b_i \ll k$. Moreover, by Lemma~\ref{lemma:C-antichains}, there are additionally at least $20(\ell-1)\binom{k+\ell}{k+1}$ antichains with $k+1$ elements that contain an element of $D_3$. Thus, using the inequality $\big(1 - \frac{\ell-1}{k-1}\big)^3 \ge 1 - 3\frac{\ell-1}{k-1}$,
\begin{equation}
  \label{eq:example-1}
  \begin{split}
    h_k(P) \binom{k+\ell}{k+1}^{-1} & \ge \left(1 - \frac{\ell-1}{k-1}\right)^3 (3k + \ell - b_2 - b_3) + \left(\frac{7}{4}\right)^{b_2} + \left(\frac{7}{4}\right)^{b_3} + 20(\ell-1) \\
    & \ge 3k+\ell-b_2-b_3 + 10(\ell-1) + \left(\frac{7}{4}\right)^{b_2} + \left(\frac{7}{4}\right)^{b_3} \\
    & \ge 3k + 1 +  \left(\frac{7}{4}\right)^{b_2} - b_2 + \left(\frac{7}{4}\right)^{b_3} - b_3.
  \end{split}
\end{equation}
Finally, using the fact that $(7/4)^b - b \ge 3/4$ for every integer $b$, we see that the right hand side of~\eqref{eq:example-1} is strictly greater than $3k+2$. This completes the analysis as $m_k(\tau_{k,n}) = (3k+2)\binom{k+\ell}{k+1}$.

\subsubsection{Almost extremal posets of type~(\ref{item:extremal-2})}

Recall that $P$ is of type~(\ref{item:extremal-2}) if and only if there is some $f \in [k+\ell-1]$ such that $|A_i| = |A_i'| = k$ for all $i \neq f$, $|A_f| = k+\ell$, and $|A_f'| \ge k$. Moreover, $|A_f'| = k$ if $f \neq 1$. In particular, $q = \ell$ and hence $S = 2k + 50\sqrt{k}\log_2k$. We now show that $h_k(P) > m_k(\tau_{k,n})$ for each such $P$, unless $\ell = 1$. Let $b = |B_{f+1}|$. By~\eqref{eq:lemma-F},
\[
\Sigma_f \ge \Sigma_{f+1} + 2|A_{f+1}'| - |D_{f+1}| - 2|B_{f+1}| \ge 3k-|D_{f+1}|-2b.
\]
As before, by Lemma~\ref{lemma:F}, we may assume that $b \le 2\ell\log_2k$ since otherwise $h_k(P) > m_k(\tau_{k,n})$. If $|D_{f+1}| < k - k^{2/3}$, then $\Sigma_1 \ge \Sigma_f \ge S$, and, by Observation~\ref{obs:Sigma-1}, $h_k(P) > m_k(\tau_{k,n})$. Thus, we may assume that $|D_{f+1}| \ge k-k^{2/3}$.

Let us now carefully count homogenous $(k+1)$-sets in $P$. First, consider the collection of all chains of length $k+1$ obtained by taking an edge $xy$ of $G_f$ and an arbitrary set of $k-1$ elements from some chain of length $k+\ell$ that contains both $x$ and $y$ (such a chain exists as $A_f' = A_f$. The number of such chains containing a fixed edge is
\[
\binom{k+\ell-2}{k-1} = \frac{(k+1)k}{(k+\ell)(k+\ell-1)} \binom{k+\ell}{k+1} \ge \left(1 - \frac{\ell-1}{k}\right)^2 \binom{k+\ell}{k+1}
\]
and all chains constructed this way are distinct. Therefore, the number $N_c$ of chains of length $k+1$ satisfies
\[
N_c \ge e(G_f) \cdot \binom{k+\ell-2}{k-1} \ge \left(1-\frac{\ell-1}{k}\right)^2(2k-b)\binom{k+\ell}{k+1}.
\]
To estimate the number of antichains, note that for any $Y \subseteq B_{f+1}$, the set $Y \cup (A_f \setminus N_{G_f}(Y))$ is an antichain with $k+\ell$ elements. Hence, the number $N_{a,1}$ of $(k+1)$-element antichains that are contained in $A_f \cup B_{f+1}$ satisfies
\[
N_{a,1} \ge 2^b \binom{k+\ell-b}{k+1-b} \ge 2^b \left(1 - \frac{\ell-1}{k+\ell-b} \right)^b  \binom{k+\ell}{k+1} \ge \left(\frac{7}{4}\right)^b \binom{k+\ell}{k+1},
\]
where the last inequality holds since $\ell, b \ll k$. Moreover, by Lemma~\ref{lemma:C-antichains}, there are additionally at least $20(\ell-1)\binom{k+\ell}{k+1}$ antichains with $k+1$ elements that contain an element of $D_{f+1}$. It follows that, using the inequality $\big(1 - \frac{\ell-1}{k}\big)^2 \ge 1 - 2\frac{\ell-1}{k}$,
\begin{equation}
  \label{eq:example-2}
  h_k(P) \binom{k+\ell}{k+1}^{-1} \ge 2k-b - 4(\ell-1) + \left(\frac{7}{4}\right)^b + 20(\ell-1).
\end{equation}
As $m_k(\tau_{k,n}) = (2k+1)\binom{k+\ell}{k+1}$, we conclude that $h_k(P) > m_k(\tau_{k,n})$ unless $\ell=1$ and $b \le 1$.

\subsubsection{Almost extremal posets of type~(\ref{item:extremal-2}) when $\ell = 1$}

We finally show that $h_k(P) \ge m_k(\tau_{k,n})$ for each poset $P$ of type~(\ref{item:extremal-2}) and provide a rough structural characterization of such posets which contain precisely $m_k(\tau_{k,n})$ homogeneous $(k+1)$-sets. Since we may assume that $\ell = 1$, then $m_k(\tau_{k,n}) = 2k+1$ and $\Sigma_1$ is the number of chains of length $k+1$ in $P$. Let $b = |B_{f+1}|$ and recall that $b \le 1$. By~\eqref{eq:lemma-F},
\[
\Sigma_1 \ge \Sigma_{f+1} + 2k - |D_{f+1}| - 2b \ge 2k - b.
\]
Moreover, $A_f \cup B_{f+1}$ contains at least $2^b$ antichains with $k+1$ elements. It follows that $h_k(P) \ge 2k - b + 2^b \ge m_k(\tau_{k,n})$. Moreover, either $h_k(P) > m_k(\tau_{k,n})$ or $\Sigma_{f+1} = k$, $D_{f+1} = A_{f+1} \setminus B_{f+1}$ (i.e., each $x \in A_{f+1} \setminus B_{f+1}$ has degree two in $G_f$), and there are only $2^b$ antichains with $k+1$ elements (plus, they are both contained in $A_f \cup B_{f+1}$). This means that for every $X \subseteq A_{f+1}$ such that $X \neq \emptyset$ and $X \neq B_{f+1}$, we have $|N_{G_f}(X)| \ge |X|+1$, as otherwise $X \cup (A_f \setminus N_{G_f}(X))$ would be an additional antichain with $k+1$ elements (other than $A_f$ and $B_{f+1} \cup (A_f \setminus N_{G_f}(B_{f+1}))$, which were already counted above in the $2^b$ term). In particular, $f=1$ and $A_1' = A_1$, as otherwise $|A_f'|=|A_{f+1}'|=k$ and consequently $|N_{G_f}(A_{f+1}')| = |A_{f+1}'|$. We now show that these conditions uniquely determine the graph $G_1$.

\begin{claim}
  Suppose that $h_k(P) = m_k(\tau_{k,n})$.
  \begin{enumerate}[(i)]
  \item
    \label{item:b0}
    If $b = 0$, then $G_1$ is a path with $2k+1$ vertices.
  \item
    \label{item:b1}
    If $b = 1$, then $G_1$ is the disjoint union of a path with $2k-1$ vertices and en edge.
  \end{enumerate}
\end{claim}
\begin{proof}
  Let $H$ be the auxiliary multigraph on $A_1$ where the multiplicity of each pair $xy$ is the number of common $G_1$-neighbors of $x$ and $y$ in $A_2$. As we have assumed that $D_2 = A_2 \setminus B_2$, it follows that $e(H) = |A_2 \setminus B_2| = k - b$. Moreover, the condition $|N_{G_1}(X)| \ge |X|+1$ implies that $H$ is a forest. (Here again we consider two parallel edges to be a cycle.) We claim that each tree in $H$ is a path. If this is not the case, then $H$ would have a vertex of degree at least $3$ and thus $G_1$ would contain the $1$-subdivision of $K_{1,3}$ as an induced subgraph. But this is not possible as both $G_1$ and its complement are comparability graphs (since $P$ is a poset of order dimension at most two) and one can check that the $1$-subdivision of $K_{1,3}$ does not have this property.

  Now, (\ref{item:b0}) follows since $b=0$ implies that $H$ is a path with $k+1$ vertices and hence $G_1$ is a path with $2k+1$ vertices. To see~(\ref{item:b1}), note first that $b=1$ implies that $H$ has two connected components. The unique vertex $z \in B_2$ has a $G_1$-neighbor in one of these components. Denote it by $C$ and let $X = N_{G_1}(C)$. One easily checks that $|X| = |N_{G_1}(X)| = |C|$ and hence $X = \{z\}$ and $|C|=1$, as we have assumed that $X = \emptyset$ and $X = B_2$ are the only subsets of $A_2$ with $|X| \le |N_{G_1}(X)|$. It follows that $H$ is the union of a path with $k$ vertices and an isolated vertex and hence $G_1$ is the union of a path with $2k-1$ vertices and an edge.
\end{proof}

Finally, the following lemma (where we take $i = 2$ and $j = k+1$) shows that $P \setminus A_1$ may be partitioned into $k$ chains, as otherwise $\Sigma_2 > k$.

\begin{lemma}
  \label{lemma:disjoint-chains}
  Suppose that $i, j \in [k+\ell]$ with $i \le j$ are such that $|A_i'| = \ldots = |A_j'| = k$. There is some $d \ge 0$ such that:
  \begin{enumerate}[(i)]
  \item
    \label{item:disjoint-chains-one}
    There are pairwise disjoint chains $L_1, \ldots, L_{k-d}$ of length $j - i + 1$ with $\min L_p \in A_i'$ and $\max L_p \in A_j'$ for each $p \in [k-d]$ and
  \item
    $\Sigma_i \ge \Sigma_j + d$.
  \end{enumerate}
\end{lemma}
\begin{proof}
  We prove the claim by reverse induction on $i$. The statement is vacuously true for $i = j$ as $|A_j'| = k$. Suppose now that $i < j$ and, appealing to the inductive assumption, let $L_1', \ldots, L_{k-d'}'$ be a collection of pairwise disjoint chains of length $j-i$ with $\min L_p \in A_{i+1}'$ and $\max L_p \in A_j'$ for each $p \in [k-d']$. Let
  \[
  X = \{\min L_p \colon p \in [k-d']\} \subseteq A_{i+1}'
  \]
  and let $k-d$ be the size of the largest $G_i$-matching between $X$ and $A_i'$. Since this matching naturally extends some $k-d$ chains in $\{L_1', \ldots, L_{k-d'}'\}$ to pairwise disjoint chains $L_1, \ldots, L_{k-d}$ satisfying assertion~(\ref{item:disjoint-chains-one}) of the lemma, it is enough to show that $\Sigma_i \ge \Sigma_{i+1} + d-d'$. By Hall's theorem, there is a $Y \subseteq X$ such that $|N_{G_i}(Y)| \le |Y| - (d-d')$. Since $G_i[A_i',A_{i+1}']$ has no isolated vertices, it follows that
  \[
  e_{G_i}(A_i', A_{i+1}') \ge e_{G_i}(A_i',Y) + e_{G_i}(A_i' \setminus N_{G_i}(Y), A_{i+1}') \ge |Y| + |A_i'| - |N_{G_i}(Y)| \ge k + d-d'.
  \]
  Consequently,
  \[
  \Sigma_i = \sum_{x \in A_{i+1}'} u_{i+1}(x) \ge \Sigma_{i+1} + e_{G_i}(A_i', A_{i+1}') - |A_{i+1}'| \ge \Sigma_{i+1} + d-d'.\qedhere
  \]
\end{proof}

This completes the proof of Theorem~\ref{thm:posets}.

\section{Concluding remarks}

\label{sec:concluding-remarks}

In this paper, we have determined the minimum number of monotone subsequences of length $k+1$ in a sequence of $n \le k^2 + k^{3/2} / (300\log_2k)$ numbers for all sufficiently large $k$. This minimum, which we denoted by $m_k(n)$, is achieved by taking $k$ increasing (decreasing) sequences of lengths $\lfloor n/k \rfloor$ or $\lceil n/k \rceil$ in such a way that there is no decreasing (increasing) subsequence of length $k+1$. One such sequence is $\tau_{k,n}$, defined in Section~\ref{sec:introduction}. Moreover, we have shown that if $n \neq k^2+k+1$, then no extremal sequence contains both increasing and decreasing subsequences of length $k+1$. Our results provide strong evidence supporting Conjecture~\ref{conj:main}, which asserts that the above statements remain true for all pairs of $k$ and $n$. It is also worth mentioning that, although we have not stated it explicitly, our proof establishes a stability statement of the following form: If a sequence of $n$ numbers is not `close' to a union of $k$ increasing (decreasing) sequences of almost equal lengths, then it contains `many' more than $m_k(n)$ monotone subsequences of length $k+1$.

Since we are still far from proving Conjecture~\ref{conj:main}, one may ask to determine at least the asymptotic behavior of the function $m_k(n)$. Let $\mu_k(n) = m_k(n) \binom{n}{k+1}^{-1}$. 
Then Conjecture~\ref{conj:main} suggests that $\mu_k(n) = (1+o(1))k^{-k}$. Standard averaging arguments can be used to show that $\mu_k(n)$ is 
non-decreasing in $n$. Therefore, the theorem of Erd{\H{o}}s and Szekeres implies that

\begin{equation}
  \label{eq:mukn-lower}
  \mu_k(n) \ge \mu_k(k^2+1) = \binom{k^2+1}{k+1}^{-1} \sim \sqrt{\frac{2\pi e}{k}} \cdot (ek)^{-k}.
\end{equation}
Our main theorem yields an improvement of~\eqref{eq:mukn-lower} by a factor of (only) $2^{\Theta(\sqrt{k})}$ and it would be interesting to improve this lower bound further.

We find very promising the prospect of studying Erd{\H{o}}s--Rademacher-type problems in other settings. In principle, one can investigate such extensions for any extremal or Ramsey-type result. Some motivation for studying these problems 
comes from the recently renewed interest in `supersaturation' results, which have been used in conjunction with the `transference' theorems of Conlon and Gowers~\cite{CoGo} and of Schacht~\cite{Sc} as well as the `hypergraph containers' theorems of 
Balogh, Morris, and the first author~\cite{BaMoSa} and of Saxton and Thomason~\cite{SaTh} to prove numerous `sparse random analogues' of classical extremal and Ramsey-type results. 

\bibliographystyle{amsplain}
\bibliography{mono-seq}

\end{document}